\documentclass[11pt,leqno,a4paper]{amsart}
\usepackage[top=2.2cm,bottom=2.2cm,right=2cm,left=2cm]{geometry}
\usepackage{amssymb,amsmath,amscd,graphicx,fontenc,amsthm,mathrsfs, mathtools}
\usepackage[frame, cmtip, arrow, matrix, line, graph, curve]{xy}
\usepackage{graphpap, color}
\usepackage[mathscr]{eucal}
\usepackage{ifthen}
\usepackage{relsize}

\usepackage{multirow}

\usepackage{hyperref}
\usepackage{cleveref}

\usepackage{fancyhdr}
\usepackage{indentfirst}
\usepackage{paralist}
\usepackage{pstricks}
\usepackage{bbm}
\usepackage{tikz}

\newtheorem{theorem}{Theorem}[section]
\newtheorem{definition}[theorem]{Definition}
\newtheorem{lemma}[theorem]{Lemma}
\newtheorem{proposition}[theorem]{Proposition}
\newtheorem{corollary}[theorem]{Corollary}

\theoremstyle{remark}
\newtheorem{remark}[theorem]{Remark}

 \theoremstyle{plain}
\newtheorem*{theorem*}{Theorem}
\newtheorem*{proposition*}{Proposition}

\numberwithin{equation}{section}
\newcommand{\EE}{{\mathbb{E}}}
\newcommand{\NN}{{\mathbb{N}}}
\newcommand{\ZZ}{{\mathbb{Z}}}

\newcommand{\RR}{{\mathbb{R}}}

\newcommand{\QQ}{{\mathbb{Q}}}

\DeclareMathOperator{\rk}{rk}

\DeclareMathOperator{\vol}{vol}

\DeclareMathOperator{\supp}{supp}

\newcommand{\diag}{\mathrm{diag}}

\newcommand{\tp}[1]{^\mathrm{t}{#1}}

\newcommand{\SL}{\mathrm{SL}}

\newcommand{\Sp}{\mathrm{Sp}}

\newcommand{\SO}{{\mathrm{SO}}}

\newcommand{\Id}{\mathrm{Id}}

\newcommand{\vx}{\mathbf x}

\newcommand{\vw}{\mathbf w}
\newcommand{\vv}{\mathbf v}

\providecommand{\ve}{\mathbf{ e}}

\newcommand{\Mat}{\mathrm{Mat}}

\newcommand{\lcm}{\mathrm{lcm}}

\renewcommand{\varpi}{\pi}

\newcommand{\origin}{\mathbf 0}

\providecommand{\Lie}[1]{\mathfrak{#1}}
\newcommand{\tr}{\mathrm{tr}}
\providecommand{\im}{\mathrm{im}}

\newcommand{\bs}[1]{\boldsymbol{#1}}
\newcommand{\Gr}{\mathrm{Gr}}

\definecolor{cmd}{rgb}{1.0, 0.35, 0.21}

\begin{document}
\title[Distribution of values of symplectic forms]{Distribution of values at tuples of integer vectors under symplectic forms}

\author{Jiyoung Han}
\email{jiyoung.han@pusan.ac.kr}
\address{Department of Mathematics Education, Pusan National University, Busan 46241, Republic of Korea}


\maketitle
\begin{abstract}
We investigate lattice-counting problems associated with symplectic forms from the perspective of homogeneous dynamics.
In the qualitative direction, we establish an analog of Margulis theorem for symplectic forms, proving density results for tuples of vectors.
Quantitatively, we derive a volume formula having a certain growth rate, and use this and Rogers' formulas for a higher rank Siegel transform to obtain the asymptotic formulas of the counting function associated with a generic symplectic form.

We further establish primitive and congruent analogs of the generic quantitative result. For the primitive case, we show that the lack of completely explicit higher moment formulas for a primitive higher rank Siegel transform does not obstruct obtaining quantitative statements.
\end{abstract}
\section{Introduction}

The  famous Margulis theorem, which is the complete answer to the Oppenheim conjecture, says that a non-degenerate and indefinite quadratic form $Q$ has the dense image set of integer vectors in $\RR$ if and only if $Q$ is not a scalar multiple of a form with integer coefficients.
He achieved the theorem using the dynamical property of the unipotent flow of $\SO(2,1)^\circ$ on the homogeneous space $\SL_3(\RR)/\SL_3(\ZZ)$ (after reduction to the case of dimension $3$) that any orbit of the unipotent flow is either unbounded or closed.
Using a similar methodology, Dani and Margulis \cite{DM1989} proved the rank-2 analog of Oppenheim conjecture, saying that for any non-degenerate and indefinite quadratic form $Q$, the image set 
\[
\left\{(Q(\vv_1), Q(\vv_2), Q(\vv_1, \vv_2)): \vv_1, \vv_2 \in \ZZ^d \right\},
\] 
where $Q(\vv_1, \vv_2)=\frac 1 2 \left(Q(\vv_1+\vv_2) - Q(\vv_1)-Q(\vv_2)\right)$ is the symmetric bilinear form corresponding to $Q$, is dense in the possible range $\left\{(Q(\vv_1), Q(\vv_2), Q(\vv_1, \vv_2)): \vv_1, \vv_2 \in \RR^d \right\}$ in $\RR^3$ if and only if $Q$ is irrational. The theorem is extended to the case of rank $k$ for $k\le d-1$ using Ratner's theorem.
Meanwhile, there has been studied qualitative Oppenheim-conjecture typed problems in many perspectives. 
Borel and Prasad \cite{BP1992} established an S-arithmetic analog, Mohammadi \cite{Mohammadi2011} obtained an analog for fields of positive characteristic except powers of $2$, and Gorodnik \cite{Gorodnik2004} proved the Oppenheim conjecture for pairs of a quadratic form and a linear form, and Sargent \cite{Sargent2014} studied the distribution of images of several linear forms at integer vectors lying on the level set of a rational non-degenerate indefinite quadartic form.

Our first main goal is considering a higher-rank Oppenheim conjecture for symplectic forms, i.e., non-degenerate skew-symmetric forms defined on $\RR^{2n}$. Let $\langle \cdot, \cdot \rangle$ be the standard symplectic form on $\RR^{2n}$ given as
\begin{equation}\label{J_n}
\langle \vv_1, \vv_2 \rangle=\vv_1^T \left(\begin{array}{cc} 
& I_n\\ -I_n \end{array}\right) \vv_2=\vv_1^T J_n \vv_2.
\end{equation}
It is well-known that any symplectic form on $\RR^{2n}$ is the scalar multiple of the conjugate of the standard symplectic form by an element of $\SL_{2n}(\RR)$, i.e., the form is given as
\[
c\langle g\vv_1, g\vv_2 \rangle
\] 
for some $c\in \RR-\{0\}$ and $g\in \SL_{2n}(\RR)$. In this article, we assume that $c=1$ and denote arbitrary (scaled) symplectic form by
$$\langle \vv_1, \vv_2 \rangle^g=\langle g\vv_1, g\vv_2 \rangle,\quad \forall \vv_1,\vv_2\in \RR^{2n},$$ 
where $g\in \SL_{2n}(\RR)$.

Following the case of quadratic forms, we will say that a symplectic form $\langle \cdot, \cdot \rangle^g$ is \emph{rational} if $\langle \cdot, \cdot \rangle^g$ is a scalar multiple of a symplectic form with integer coefficients, and \emph{irrational} if $\langle \cdot, \cdot \rangle^g$ is not rational.

\begin{theorem}\label{thm: qualitative} Let $n\ge 2$. Assume that $2\le k\le n+2$ for $n\ge 3$ and $k=2,\;3$ when $n=2$.
Let $\langle \cdot, \cdot \rangle^g$ be a symplectic form on $\RR^{2n}$ associated with $g\in \SL_{2n}(\RR)$. The image set
\[
\left\{(\langle\vv_i, \vv_j\rangle^g)_{1\le i<j\le k}: \vv_1, \ldots, \vv_k \in \ZZ^{2n}\right\}
\]
is dense in $\RR^{\frac 1 2 k(k-1)}$ if and only if $\langle \cdot, \cdot \rangle$ is not a scalar multiple of any symplectic forms with integer coefficients.
\end{theorem}

When $k=1$, since any symplectic form is skew-symmetric, only possible value for a symplectic form at any integer vector (or any vector) is $0$, hence Theroem~\ref{thm: qualitative} does not exist.
When $k=2$, any symplectic form on $\RR^{2n}\times \RR^{2n}$ can be viewed as the quadratic form on $\RR^{4n}$ of signature $(2n,2n)$.
Thus, when $k=2$, \Cref{thm: qualitative} is the special case of Margulis theorem.

In \cite{EMM1998}, Eskin, Margulis and Mozes found conditions of quadratic forms $Q$ of signature $(p,q)$, where $p\ge 3$ and $q\ge1$, that have the asymptotic formula of the number of integer vectors $\vv$ such that  $\|\vv\|<T$ and $Q(\vv)\in (a,b)$ for a given $(a,b)\subseteq \RR$, as $T\rightarrow \infty$. See \cite{EMM2005} for quadratic forms of signature $(2,2)$, \cite{WKim2024} of signature $(2,1)$. In \cite{MM2011}, Margulis and Mohammadi obtained a similar result for inhomogeneous quadratic forms.
For other types of Oppenheim-conjecture problems, Sargent \cite{Sargent2014b} quantified his result, Han, Lim and Mallahi-Karai \cite{HLMK2017} proved the quantitative version of the theorem of Borel--Prasad on the $S$-arithmetic space, and the same authors obtained the quantitative result of the Gorodnik theorem for a pair of a quadratic form and a linear form. 

If we relax the question to consider asymptotic formulas in a generic sense, one can formulate a broader family of problems in the quantitative Oppenheim framework, and even obtain information on error terms. Most of these results are using Siegel transforms and the first and second moment formulas of them, combining with probabilistic methods.
Athreya and Margulis \cite{AM2018} derived the power-saving error bound of the asymptotic formula for generic quadratic forms, and Kelmer and Yu \cite{KY2020} established a similar result for generic homogeneous polynomials of degree $d\ge 2$ on certain orbits. For inhomogeneous cases, there are results of Marklof \cite{Marklof2003} for the case of dimension 4, and of Ghosh, Kelmer and Yu \cite{GKY2022, GKY2023}. The author \cite{Han2022} computed the higher moment formulas for an $S$-arithmetic Siegel transform and derived the $S$-arithmetic random quantitative result. In \cite{BGH2021}, Bandi, Ghosh and the author proved the random quantitative Oppenheim conjecture for a system of a quadratic form and several linear forms.
See also \cite{Hanquandratic} for the random quantitative analog of the theorem of Dani and Margulis, considering values of a quadratic form at pairs of two integer vectors. 

Our next result is the random quantitative result of Theorme~\ref{thm: qualitative}.

\begin{theorem}\label{thm: quantitative} Assume that $k\ge 2$ and $2n \ge k^2+3$.
Let $\mathcal I=\left\{(a_{ij}, b_{ij})\right\}_{1\le i<j\le k}$ be a collection of any bounded intervals in $\RR$.
For a symplectic form $\langle \cdot, \cdot \rangle^g$, denote
\[
N_{g,\mathcal I}(T)=\#\left\{(\vv_1, \ldots, \vv_k)\in (\ZZ^{2n})^k: \begin{array}{c}
\langle \vv_i,\vv_j\rangle^g\in (a_{ij}, b_{ij})\;\text{for}\; 1\le i<j\le k;\\
\|\vv_\ell\|< T\;\text{for}\;1\le \ell\le k
\end{array}\right\}.
\]
There is $\delta_0>0$ so that for any $\delta\in (\delta_0, 1)$, it follows that for almost all $g$,
\[
N_{g,\mathcal I}(T)=c_g \prod_{1\le i<j\le k} (b_{ij}-a_{ij})\cdot T^{2nk-k(k-1)} + O_{g,\mathcal I}\left(T^{\delta(2nk-k(k-1))}\right)
\]
for some constant $c_g>0$ depending only on $g$.
\end{theorem}

The main term of the asymptotic formula is an estimate of the volume of the region where those $(\vv_1, \ldots, \vv_k)\in (\RR^{2n})^k$ satisfying that $\langle \vv_i, \vv_j\rangle^g \in (a_{ij}, b_{ij})$ and $\|\vv_\ell\|<T$.

\begin{theorem}\label{thm: volume formula}
Assume that $n\ge 2$ and $2\le k\le 2n-2$. Let $\langle \cdot, \cdot \rangle^g$ be any symplectic form on $\RR^{2n}$.
Choose $(a_{ij},b_{ij})\subseteq \RR$ for $1\le i<j \le k$ such that for any $(t_{ij})_{i,j}$ with $a_{ij}<t_{ij}<b_{ij}$, there is $(\vv_1, \ldots, \vv_k)\in (\RR^{2n})^k$ for which $\langle \vv_i,\vv_j\rangle^g=t_{ij}$ for all $1\le i<j\le k$.
There is $T_0>0$ depending on $g$ and $(a_{ij}, b_{ij})$'s such that if $T>T_0$, it holds that
\[\begin{split}
&\vol\left(\left\{(\vv_1, \ldots, \vv_k)\in (\RR^{2n})^k: \begin{array}{c}
\langle \vv_i,\vv_j \rangle^g\in (a_{ij}, b_{ij})\;\text{for}\;1\le i<j\le k;\\[0.05in]
\|\vv_\ell\|<T\;\text{for}\;1\le \ell \le k \end{array}\right\}\right)\\
&\hspace{1in} =c_g \prod_{1\le i<j\le k} (b_{ij}-a_{ij}) \cdot T^{2nk - k(k-1)} + \left\{\begin{array}{cl}
O_g\left(T^{2nk-k(k-1)-1}\right),&\text{if }2\le k\le 2n-3;\\[0.05in]
O_g\left(T^{2nk-k(k-1)-\frac 1 2}\right),&\text{if }k=2n-2.\end{array}\right.
\end{split}\]
\end{theorem}

Let $P(\ZZ^{2n})$ be the set of primitive integer vectors in $\RR^{2n}$. It is straightforward to observe that a symplectic form satisfying the same condition in Theorem~\ref{thm: qualitative} also has a dense image when restricted to tuples of primitive vectors.
For the primitive analog of Theorem~\ref{thm: quantitative}, we will make use of higher moment formulas for the primitive Siegel transform, although the complete form of such formulas is unknown.

\begin{theorem}\label{thm: quantitative prime} Let $k, n\in \NN$ and $\mathcal I$ be as in \Cref{thm: quantitative}.
For a symplectic form $\langle\cdot, \cdot\rangle^g$, denote
\[
\widehat{N}_{g,\mathcal I}(T)=\#\left\{(\vv_1, \ldots, \vv_k)\in P(\ZZ^{2n})^k: \begin{array}{c}
\langle \vv_i,\vv_j\rangle^g \in (a_{ij}, b_{ij})\;\text{for}\; 1\le i<j\le k;\\
\|\vv_i\|< T\;\text{for}\;1\le i\le k
\end{array}\right\}.
\]
Let $\delta_0>0$ be as in \Cref{thm: quantitative} and fix $\delta\in (\delta_0, 1)$. For almost all $g\in \SL_{2n}(\RR)$, it holds that
\[
\widehat{N}_{g,\mathcal I}(T)=\frac{c_g}{\zeta(2n)^k}\prod_{1\le i<j\le k} (b_{ij}-a_{ij})\cdot T^{2nk-k(k-1)} + O_{g,\mathcal I}\left(T^{\delta(2nk-k(k-1))}\right),
\]
where $c_g>0$ is as in \Cref{thm: quantitative} and $\zeta(\cdot)$ is the Reimann-zeta function.
\end{theorem}

In \cite{AGH2024}, Alam, Ghosh and the author showed the higher moment formulas for a Siegel transform having a congruence condition. Using their formulas, one can show the following theorem. Recall that for $N\in \NN$ and $\vv_0\in \ZZ^{2n}$, 
\[
\vv\equiv \vv_0\mod N
\quad\text{if and only if}\quad
\vv\in \vv_0+N\ZZ^{2n}.
\]

\begin{theorem}\label{thm: quantitative congruence}
    Let $k, n\in \NN$ and $\mathcal I$ be as in \Cref{thm: quantitative}.
    Let $N\in \NN$ and $\vv_0\in \ZZ^{2n}$ with $\gcd(\vv_0, N)=1$ be given.
For a symplectic form $\langle\cdot, \cdot\rangle^g$, denote
\[
{N}_{g,\mathcal I,\vv_0, N}(T)=\#\left\{(\vv_1, \ldots, \vv_k)\in (\vv_0+N\ZZ^{2n})^k: \begin{array}{c}
\langle \vv_i,\vv_j\rangle^g \in (a_{ij}, b_{ij})\;\text{for}\; 1\le i<j\le k;\\
\|\vv_i\|< T\;\text{for}\;1\le i\le k
\end{array}\right\}.
\]
Let $\delta_0>0$ be as in \Cref{thm: quantitative} and fix $\delta\in (\delta_0, 1)$. For almost all $g\in \SL_{2n}(\RR)$, it holds that
\[
{N}_{g,\mathcal I,\vv_0, N}(T)=\frac{c_g}{N^{2nk}}\prod_{1\le i<j\le k} (b_{ij}-a_{ij})\cdot T^{2nk-k(k-1)} + O_{g,\mathcal I}\left(T^{\delta(2nk-k(k-1))}\right),
\]
where $c_g>0$ is as in \Cref{thm: quantitative}.
\end{theorem}

%
%
%

\subsection*{Organization}
In Section~\ref{Sec:Oppenheim Conjecture for Symplectic Forms}, we establish an analog of Margulis theorem for symplectic forms. 
 We first obtain the result in full generality for the rank range $2\le k \le 2n-1$, showing density in the region $$\left\{(\langle\vv_i, \vv_j\rangle^g)_{1\le i<j\le k}: \vv_1, \ldots, \vv_k \in \RR^{2n}\right\}$$ (see Theorem~\ref{thm: qualitative 2}). We show that this region coincides with $\RR^{\frac 1 2 k(k-1)}$ when $k\le n+2$.
In Section~\ref{Sec:Volume Formula}, we prove Theorem~\ref{thm: volume formula}. To deduce the volume estimate appearing in Theorem~\ref{thm: volume formula}, as well as the growth rate in $T$, we decompose the Lebesgue measure on $\RR^{2n}$ into the measure supported on the \emph{rank-k-cone} associated with the given symplectic form and the $k(k-1)/2$-number of one-dimensional measures corresponding to values of the symplectic form.
In Section~\ref{Sec:Use of Rogers' Formulas}, we apply Theorem~\ref{thm: volume formula} and Rogers' formulas to prove Theorem~\ref{thm: quantitative}. The condition $2n\ge k^2+3$ guarantees an upper bound (appropriate for our purpose) on the difference between the first and the second moments of the rank-$k$ Siegel transform (see Definition~\ref{def: Siegel transform}). For this, we refine the results of Rogers in \cite{Rogers55B} (see Theorem~\ref{prop: upper bound}).
In Section~\ref{Sec:Primitive Analog}, we prove primitive and congruent analogs of Theorem~\ref{thm: quantitative}. For the primitive case, in contrast to the non-primitive case or congruent case, the moment formulas for the rank-$k$ primitive Siegel transform are not known in a completely explicit form if $k\ge 2$. Nevertheless, we show that this does not present a genuine obstruction for deducing Theorem~\ref{thm: quantitative prime}. 
In the Appendix, we contain the proof of the well-known fact that the symplectic group $\Sp({2n},\RR)$ is a maximal connected subgroup of $\SL_{2n}(\RR)$ using the restricted root system of $\Lie{sp}(2n,\RR)$ to the adjoint representation of $\Lie{sl}_n(\RR)$.

\subsection*{Acknowledgment}
I would like to thank Anish Ghosh and Seungki Kim for valuable advice and discussions. This work is supported by the National Research Foundation of Korea (NRF) grant funded by the Korea government (project No. RS-2025-00515082). 

\section{Oppenheim Conjecture for Symplectic Forms}\label{Sec:Oppenheim Conjecture for Symplectic Forms}

\begin{theorem}\label{thm: qualitative 2} Let $n\ge 2$ and $1\le k\le 2n-1$.
Let $\langle\cdot, \cdot\rangle^g$ be a symplectic form on $\RR^{2n}$ associated with $g\in \SL_{2n}(\RR)$. The image set
\[
\left\{\left(\langle\vv_i, \vv_j\rangle^g\right)_{1\le i<j\le k}: \vv_1, \ldots, \vv_k \in \ZZ^{2n}\right\}
\]
is dense in $\{(\langle\vv_i, \vv_j\rangle^g)_{1\le i<j\le k}: \vv_1, \ldots, \vv_k \in \RR^{2n}\}\subseteq \RR^{\frac 1 2 k(k-1)}$ if and only if $\langle \cdot, \cdot\rangle^g$ is irrational.
\end{theorem}

The proof proceeds by the canonical argument, namely Ratner's orbit closure theorem~\cite{Ratner1991}.

\begin{theorem}[Ratner's Orbit Closure Theorem]\label{thm: Ratner}
Let $G$ be a connected Lie group and $\Gamma$ a lattice subgroup. Let $H$ be a connected Lie subgroup of $G$ generated by $\mathrm{Ad}$-unipotent one-parameter subgroups. 

For any $x\in G/\Gamma$, there is a Lie subgroup $L\le G$, containing $H$, such that $\overline{H.x}=L.x$ and $L.x$ carries a finite $H$-invariant measure.
\end{theorem}

Denote $x=g\Gamma$. The theorem further says that $L \cap g\Gamma g^{-1}$ is a lattice subgroup of $L$.
From now on, denote $G_{d}=\SL_{d}(\RR)$ and $\Gamma_{d}=\SL_{d}(\ZZ)$ for $d\in \NN_{\ge 2}$.
To connect the density of $\Sp(2n,\RR)$-orbits in $G_{2n}/\Gamma_{2n}$ with arithmetic properties of corresponding symplectic forms, we need the following theorem.

\begin{theorem}\cite[Borel Density Theorem]{Borel1960}
Let $G$ be a connected semisimple algebraic group over $\RR$ without compact factors.
Then any lattice subgroup $\Gamma\le G$ is Zariski dense in $G$
\end{theorem}

\begin{proposition}\label{prop: maximal cnt subgp}
The symplectic group $\Sp(2n,\RR)$ is maximal among proper connected Lie subgroups of $\SL_{2n}(\RR)$ for $n\ge 2$.
\end{proposition}
\begin{proof}
The result is a direct consequence of the representation theory of semisimple Lie algebra. For the sake of completeness, we contain the proof in the Appendix.
\end{proof}

\begin{proof}[Proof of Theorem~\ref{thm: qualitative 2}]
Let $\langle \cdot, \cdot \rangle^g$, $g\in G_{2n}=\SL_{2n}(\RR)$ be a symplectic form. The conjugate $\Sp(2n,\RR)^g=g^{-1}\Sp(2n,\RR)g$ of the symplectic group is the subgroup of $G_{2n}$ preserving the given symplectic form $\langle \cdot, \cdot \rangle^g$.

Since $\Sp(2n,\RR)^g$ is generated by unipotents, it follows from Ratner's orbit closure theorem and Proposition~\ref{prop: maximal cnt subgp} that the orbit closure $\overline{\Sp(2n,\RR)^g.\Gamma_{2n}/\Gamma_{2n}}$ in $G_{2n}/\Gamma_{2n}$ is either
\[
\Sp(2n,\RR)^g.\Gamma_{2n}/\Gamma_{2n}
\;\text{or}\; G_{2n}/\Gamma_{2n}.
\]
We remark that if the former case holds, $\Sp(2n,\RR)^g\cap \Gamma_{2n}$ is a lattice subgroup of $\Sp(2n,\RR)^g$, and if the latter case holds,
\[\begin{split}
\overline{\left\{\left(\langle\vv_i,\vv_j\rangle^g\right)_{1\le i<j\le k}: \vv_1, \ldots, \vv_k\in \ZZ^{2n}\right\}}
&=\overline{\bigcup_{h\in \Sp(2n,\RR)^g\:\Gamma_{2n}}\left\{\left(\langle h\vv_i,h\vv_j\rangle^g\right)_{1\le i<j\le k}: \vv_1, \ldots, \vv_k\in \ZZ^{2n}\right\}}\\
&\supseteq \bigcup_{h\in \overline{\Sp(2n,\RR)^g\:\Gamma_{2n}}}\left\{\left(\langle h\vv_i,h\vv_j\rangle^g\right)_{1\le i<j\le k}: \vv_1, \ldots, \vv_k\in \ZZ^{2n}\right\}\\
&\supseteq \bigcup_{h\in G_{2n}}\left\{\left(\langle h\vv_i,h\vv_j\rangle^g\right)_{1\le i<j\le k}: \vv_1, \ldots, \vv_k\in \ZZ^{2n}\right\}\\
&=\left\{\left(\langle \vv_i,\vv_j\rangle^g\right)_{1\le i<j\le k}: \vv_1, \ldots, \vv_k\in \RR^{2n}\right\}.
\end{split}\]

Hence it suffices to show that
\[
\langle \cdot, \cdot\rangle^g\text{ is a rational symplectic form}
\;\Leftrightarrow\;
\overline{\Sp(2n,\RR)^g.\Gamma_{2n}/\Gamma_{2n}}=\Sp(2n,\RR)^g\Gamma_{2n}/\Gamma_{2n}.
\]

Suppose that $\langle \cdot, \cdot\rangle$ is a rational form, i.e., a scalar multiple of some symplectic form with integer coefficients. Clearly, $\{(\langle \vv_i,\vv_j\rangle^g)_{1\le i<j\le k}: \vv_1, \ldots, \vv_k\in \ZZ^{2n}\}$ is a discrete set in $\RR^{\frac 1 2 k (k-1)}$. It follows from the observation above that the orbit $\Sp(2n,\RR)^g.\Gamma_{2n}/\Gamma_{2n}$ is closed in $G_{2n}/\Gamma_{2n}$.

For the reverse direction, we will use the fact that $\Sp(2n,\RR)^g\cap \Gamma_{2n}$ is a lattice subgroup of $\Sp(2n,\RR)^g$.
From the Borel density theorem, since $\Sp(2n,\RR)^g$ is a connected semisimple algebraic group without compact factor, $\Sp(2n,\RR)^g\cap \Gamma_{2n}$ is a Zariski dense subset, and hence $\Sp(2n,\RR)^g$ is defined over $\QQ$.

Observe that if two symplectic forms $\langle \cdot, \cdot \rangle_1$ and $\langle \cdot, \cdot\rangle_2$ have the common symplectic group, they only differ by a scalar multiplication. 
Indeed, let $\langle \vv_1,\vv_2\rangle_1=c_1\langle \vv_1, \vv_2\rangle^{g_1}$ and $\langle \vv_1, \vv_2\rangle_2=c_2\langle \vv_1, \vv_2\rangle^{g_2}$. It follows that $$(g_1^{-1}g_2)^{-1}\Sp(2n,\RR)(g_1^{-1}g_2)=\Sp(2n,\RR),$$ i.e., $g_1^{-1}g_2\in C_{G_{2n}}(\Sp(2n,\RR))=\{\pm I_{2n}\}$. Thus $\langle \cdot, \cdot\rangle_2=\frac {c_2}{c_1}\langle \cdot, \cdot \rangle_1$.

Now, let us take a scalar $c\in \RR$ so that $c\langle \cdot, \cdot \rangle^g$ has at least one rational coefficient.

Consider an automorphism $\phi$ of $\RR/\QQ$ and use the same notation $\phi$ for the extension maps on the space of symplectic forms and $\Mat_{2n}(\RR)$. It is easy to show that
\[
\Sp(\phi(c\langle \cdot, \cdot\rangle^g))=\phi\Sp(c\langle \cdot, \cdot\rangle^g)=\Sp(c\langle \cdot, \cdot\rangle^g),
\]
where $\Sp(c\langle \cdot, \cdot \rangle^g)=\Sp(\langle \cdot, \cdot \rangle^g)=\Sp(2n,\RR)^g$ and $\phi(c\langle \cdot, \cdot\rangle^g)$ is the symplectic form obtained by applying $\phi$ to coefficients of the form $c\langle \cdot, \cdot \rangle^g$.
Since the form $c\langle \cdot, \cdot \rangle^g$ contains a rational coefficient, it follows that $\phi (c\langle \cdot, \cdot\rangle^g)=c\langle \cdot, \cdot \rangle^g$.

Since $\phi$ is arbitrary and one can take an automorphism of $\RR/\QQ$ sending one irrational number to another, we conclude that $\langle \cdot, \cdot \rangle^g$ is a rational form. 
\end{proof}

\begin{proof}[Proof of Theorem~\ref{thm: qualitative} assuming Theorem~\ref{thm: qualitative 2}.]
We need to show that if $k=2,3$ for $n=2$ and $2\le k\le n+2$, the set
\[
\left\{\left(\langle \vv_i,\vv_j\rangle^g\right)_{1\le i<j\le k}: \vv_1, \ldots, \vv_k\in \RR^{2n}\right\}\subseteq \RR^{\frac 1 2 k(k-1)}
\]
contains the dense subset of $\RR^{\frac 1 2 k (k-1)}$. We may assume that $g=I_{2n}$. Precisely, we will show that for any $(\xi_{ij})_{1\le i<j\le k}$, where $\xi_{ij}\neq 0$ for all $1\le i<j\le k$, one can find $\vv_1, \ldots, \vv_k\in \RR^{2n}$ such that $\langle \vv_i, \vv_j\rangle =\xi_{ij}$ for all $1\le i<j\le k$. 

Suppose that such a tuple $(\xi_{ij})_{1\le i<j\le k}$ is given. Choose any nonzero $\vv_1\in \RR^{2n}$. We want to find $\vv_2\in \RR^{2n}$ such that $\langle \vv_1, \vv_2\rangle=\xi_{12}$, and $\vv_1$ and $\vv_2$ are linearly independent.
Let $L_1:\RR^{2n}\rightarrow \RR$ be a linear map given as
\[
L_1(\vx)={\tp{\vv_1}}J_n\vx,\;\vx\in \RR^{2n},
\]
where $J_n$ is the standard skew-symmetric matrix given as in \eqref{J_n}.
Since $\vv_1\neq 0$, it follows that 
\[\rk(\im L_1)=1
\quad\text{and}\quad 
\dim\ker L_1=2n-1.\] 
Take any vector $\vv'_2\in \RR^{2n}$ such that $\langle\vv_1, \vv'_2\rangle=\xi_{12}$. For any $\vv\in \ker L_1 +\vv'_2$, $\langle \vv_1, \vv_2 \rangle=\xi_{12}$. Choose $\vv_2\in \ker L_1+\vv'_2$ such that $\vv_1$ and $\vv_2$ are linearly independent. 

We claim that one can find linearly independent $\vv_1, \ldots, \vv_\ell\in \RR^{2n}$, if $\ell\le \min(n+1, k)$, such that $\langle \vv_i, \vv_j\rangle=\xi_{ij}$ for all $1\le i<j\le \ell$. 
Assume that there are linearly independent $\vv_1, \ldots, \vv_{\ell'}\in \RR^{2n}$ such that $\langle\vv_i, \vv_j\rangle=\xi_{ij}$ for all $1\le i<j\le \ell'$ ($\ell'<n$). Take a linear map $L_{\ell'}:\RR^{2n}\rightarrow \RR^{\ell'}$ by
\begin{equation}\label{eqn: linear map}
L_{\ell'}(\vx)=\left(\begin{array}{c}
{\tp{\vv_1}}\\
\vdots\\
{\tp{\vv_{\ell'}}}\end{array}\right) J_n \vx. 
\end{equation}
Note that $\rk L_{\ell'}=\ell'$ so that $\im L_{\ell'}=\RR^{\ell'}$ and $\dim\ker L_{\ell'}=2n-\ell'$. Fix any $\vv'_{\ell'+1}\in \RR^{2n}$ such that $\langle\vv_i, \vv'_{\ell'+1}\rangle=\xi_{i, \ell'+1}$ for all $1\le i \le \ell'$. Then any $\vv_{\ell'+1}\in \ker L_{\ell'}+\vv'_{\ell'+1}$ satisfies the same property that $\langle \vv_i, \vv_{\ell'+1}\rangle=\xi_{i,\ell'+1}$ for all $1\le i\le \ell'$ and since $\vv'_{\ell'+1}\neq 0$ (from the assumption that $\xi_{i,\ell+1}\neq 0$), there is no subspace of dimension $2n-\ell'$ containing an affine subspace $\ker L_{\ell'}+\vv'_{\ell'+1}$.

To find $\vv_{\ell'+1}\in \ker L_{\ell'}+\vv'_{\ell'+1}$ for which $\vv_1, \ldots, \vv_{\ell'}, \vv_{\ell'+1}$ are linearly independent, we need that
\[
\ker L_{\ell'}+\vv'_{\ell'+1} \not\subseteq \text{$\RR$-span of}\;\vv_1, \ldots, \vv_{\ell'}
\]
which follows from the fact that $2n-\ell'\le \ell'$.

Until now, we have obtained the theorem for $2\le k\le n+1$ with an extra property that $\vv_1, \ldots, \vv_k$ are linearly independent. When $n\ge 3$ and $k=n+2$, the linear map $L_{n+1}$ defined as in \eqref{eqn: linear map} is onto, hence one can find $\vv_{n+2}$ (with above notation, $\vv'_{n+2}$) for which  $\langle\vv_i, \vv_j\rangle=\xi_{ij}$ for all $1\le i<j\le n+2$. 
\end{proof}

\section{Volume Formula}\label{Sec:Volume Formula}

Define \emph{the rank-$k$ cone} $\mathcal C_{g,k}$ of the symplectic form $\langle \cdot, \cdot\rangle^g$ in $(\RR^{2n})^k$, where $1\le k \le 2n$, by
\[
\mathcal C_{g,k}
=\left\{(\vv_1, \ldots, \vv_k)\in \RR^{2n}\times \RR^{2n}: \langle\vv_i, \vv_j\rangle^g=0, 1\le \forall i, j \le k \right\}.
\]

\begin{proposition}\label{prop: volume formula}
Let $\langle \cdot, \cdot\rangle$ be the standard symplectic form on $\RR^{2n}$ given as in \eqref{J_n}, where $n\ge 2$, and let $1\le k\le 2n-2$. 

Let $\left\{(a_{ij},b_{ij})\subseteq \RR\right\}_{1\le i<j\le k}$ be a collection of bounded intervals such that for each $t_{ij}\in (a_{ij}, b_{ij})$, there is $(\vv_1, \ldots, \vv_k)\in (\RR^{2n})^k$ such that $\langle \vv_i, \vv_j\rangle=t_{ij}$.
Let $h_\ell$ be a smooth function on $\RR^{2n}$ for $1\le \ell \le k$ with $\supp h_\ell\subseteq B_R(\origin)$ for some $R>0$, where $B_R(\origin)$ is the ball of radius $R$ centered at the origin. 

Set $N>0$ such that all $(a_{ij}, b_{ij})\subseteq [-N,N]$.
Denote by $\chi^{}_{ij}$ the characteristic function of $(a_{ij}, b_{ij})$ for $1\le i<j\le k$.
It follows that
\[\begin{split}
&\int_{\RR^{2n}}\cdots \int_{\RR^{2n}}
\prod_{\ell=1}^k h_\ell \left(\frac {\vv_\ell} {T}\right) \prod_{1\le i<j\le k}\chi^{}_{ij}(\langle \vv_i, \vv_j\rangle) d\vv_k \cdots d\vv_1\\
&\hspace{1.4in}=J(h_1, \ldots, h_k) \prod_{1\le i<j\le k} (b_{ij}-a_{ij})\cdot T^{2nk - k(k-1)}\\
&\hspace{1.6in}+O\left(\mathcal S_0 T^{2nk-k(k-1)-(2n-k-1)}\right)+ \sum_{t=1}^{k-1} O\left(\mathcal S_t T^{2nk-k(k-1)-2t}\right),
\end{split}\]
where 
\[
J(h_1, \ldots, h_k)
=\int_{\mathcal C_{I,k}} \prod_{\ell=1}^k h_{\ell} (\vw_\ell) \prod_{\ell=1}^{k-1} \|\vw_\ell\|^{k-\ell} d\vw_k \cdots d\vw_1.
\]

Here, $\mathcal S_0=\prod_{\ell=1}^k \|h_\ell\|_\infty$ and 
\[
\mathcal S_t:=\max\left\{\prod_{j\in I^c} \|h_j\|_\infty \prod_{i\in I} \mathcal S_1(h_i): I\subseteq \{1, \ldots, k\},\;|I|=t \right\},\quad 1\le t\le k-1,
\]
where $\mathcal S_1(h)=\max\{\|\partial h/\partial x_i\|_\infty: i=1, \ldots, 2n\}$ for a smooth function $h$ on $\RR^{2n}$.

The implicit constants of error terms can be taken continuously on $R$ and $N>0$.
\end{proposition}

\begin{proof}
For clarity, we present the proof in the case $k=3$. The general case follows by the same method, although the notation and steps become considerably heavier (the case $k=2$ is straightforward).

By the change of variables $\vv_\ell$ to $T\vv_\ell$, $1\le \ell \le 3$, what we want to estimate is the integral
\begin{equation}\label{main volume 1}
\int_{\RR^{2n}}\int_{\RR^{2n}}\int_{\RR^{2n}}
\prod_{\ell=1}^3 h_\ell(\vv_\ell)
\prod_{1\le i<j\le 3} \chi^{}_{ij}(T^2\langle \vv_i, \vv_j\rangle)
(T^{2n})^3 d\vv_3 d\vv_2 d\vv_1.
\end{equation}

\noindent {\bf Step 1.}\quad 
Set the first error term (bound) by cutting out a small ball centered at the origin in the first $\RR^{2n}$. 
\[
\mathcal E_1:=
\int_{B_{\frac 1 T}(\origin)}\int_{\RR^{2n}}\int_{\RR^{2n}}
\mathcal S_0 \prod_{1\le i<j\le 3} \chi^{}_{ij} (T^2 \langle \vv_i, \vv_j\rangle) (T^{2n})^3 d\vv_3d\vv_2d\vv_1
\]
so that 
\[
\eqref{main volume 1}=\int_{\RR^{2n}\setminus B_{\frac 1 T}(\origin)} \int_{\RR^{2n}}\int_{\RR^{2n}}
\prod_{\ell=1}^3 h_{\ell}(\vv_\ell) \prod_{1\le i<j\le 3} \chi^{}_{ij} (T^2 \langle \vv_i, \vv_j\rangle) (T^{2n})^3 d\vv_3 d\vv_2 d\vv_1
+\mathcal E_1.
\]

\vspace{0.1in}
\noindent {\bf Step 2.}\quad
For each $\vv_1\neq \mathbf 0$, define
\[\begin{gathered}
\boldsymbol{\xi}_{12}=\boldsymbol{\xi}_{13}=\frac {-J_n\vv_1} {\|\vv_1\|^2} \in \RR^{2n},\\
W_{12}=W_{13}=\ker \left[\langle\vv_1, \cdot\rangle:\RR^{2n}\rightarrow \RR\right] \subseteq \RR^{2n},
\end{gathered}\]
where $J_n$ is as in \eqref{J_n}.

Note that $\bs{\xi}_{1\ell}$ and $W_{1\ell}$ ($\ell=2,3$), respectively, are regarded as continuous maps $\bs{\xi}_{1\ell}(\vv_1)$ and  $W_{1\ell}(\vv_1)$ from $\RR^{2n}\setminus \{\origin\}$ to $\RR^{2n}$ and $\Gr_{2n-1}(\RR^{2n})$, respectively, where $\Gr_t(\RR^{2n})$ is a Grassmannian of $t$-dimentional subspaces of $\RR^{2n}$. 

One can decompose $\vv_\ell= t_{1\ell}\bs{\xi}_{1\ell}+\vw_{1\ell}$, where $\vw_{1\ell}\in W_{1\ell}$ ($\ell=2,3$). It is not difficult to check that $d\vv_\ell=\|\vv_1\|d\vw_{1\ell} d\bs{\xi}_{1\ell}$ so that
\[
d\vv_3d\vv_2d\vv_1=\|\vv_1\|^2 (d\vw_{13} d\bs{\xi}_{13}) (d\vw_{12} d\bs{\xi}_{12}) d\vv_1.
\]
Hence
\[\begin{split}
\eqref{main volume 1}&=
\int_{\RR^{2n}\setminus B_{\frac 1 T}(\origin)} \int_{\RR.\bs{\xi}_{12}} \int_{W_{12}}  \int_{\RR.\bs{\xi}_{13}} \int_{W_{12}}h_1(\vv_1) \prod_{\ell=2}^3 h_\ell (t_{1\ell}\bs{\xi}_{1\ell} +\vw_{1\ell})\prod_{\ell=2}^3 \chi^{}_{1\ell} (T^2t_{1\ell})\\
&\hspace{0.5in} 
\times \chi^{}_{23} \left(T^2 \langle t_{12}\bs{\xi}_{12}+\vw_{12}, t_{13}\bs{\xi}_{13}+\vw_{13}\rangle\right)
\|\vv_1\|^2 (T^{2n})^3 (d\vw_{13} d\bs{\xi}_{13}) (d\vw_{12} d\bs{\xi}_{12}) d\vv_1+\mathcal E_1.
\end{split}\]

\vspace{0.1in}
\noindent {\bf Step 3.}\quad
The range of $t_{1\ell}$ ($\ell=2,3$) is
\[
t_{1\ell}\in \left[\frac {a_{1\ell}} {T^2}, \frac {b_{1\ell}} {T^2}\right] \subseteq \left[ -\frac {N} {T^2}, \frac {N} {T^2} \right].
\]
Hence there is $T_0>0$, depending continuously on $N$ (and $h_\ell$), such that for $T>T_0$, 
\[\begin{gathered}
\|t_{1\ell}\bs{\xi}_{1\ell}\|\le N/(\|\vv_1\|T^2)\le N/T \ll1\;\text{so that}\\
h_\ell (t_{1\ell}\bs{\xi}_{1\ell}+\vw_{1\ell})=h_{\ell} (\vw_{1\ell})+ O\left(\frac {\mathcal S_1(h_\ell)} {\|\vv_1\|T^2} \right),\quad \ell=1,2.
\end{gathered}\]
Define
\[\begin{split}
\mathcal E_2&=2\int_{\RR^{2n}\setminus B_{\frac 1 T}(\origin)}\int_{\RR.\bs{\xi}_{12}}\int_{W_{12}}
\int_{\RR.\bs{\xi}_{13}}\int_{W_{13}} \mathcal S_1\prod_{\ell=2}^3 \chi^{}_{1\ell} (T^2t_{1\ell})\\
&\hspace{0.5in} 
\times \chi^{}_{23} \left(T^2 \langle t_{12}\bs{\xi}_{12}+\vw_{12}, t_{13}\bs{\xi}_{13}+\vw_{13}\rangle\right)
\|\vv_1\| T^{6n-2}(d\vw_{13} d\bs{\xi}_{13}) (d\vw_{12} d\bs{\xi}_{12}) d\vv_1\\
&+\int_{\RR^{2n}\setminus B_{\frac 1 T}(\origin)}\int_{\RR.\bs{\xi}_{12}}\int_{W_{12}}
\int_{\RR.\bs{\xi}_{13}}\int_{W_{13}} \mathcal S_2\prod_{\ell=2}^3 \chi^{}_{1\ell} (T^2t_{1\ell})\\
&\hspace{0.5in} 
\times \chi^{}_{23} \left(T^2 \langle t_{12}\bs{\xi}_{12}+\vw_{12}, t_{13}\bs{\xi}_{13}+\vw_{13}\rangle \right)
 T^{6n-4}(d\vw_{13} d\bs{\xi}_{13}) (d\vw_{12} d\bs{\xi}_{12}) d\vv_1.
\end{split}\]
It follows that for $T>T_0$, 
\[\begin{split}
\eqref{main volume 1}&=
\int_{\RR^{2n}\setminus B_{\frac 1 T}(\origin)} \int_{\RR.\bs{\xi}_{12}} \int_{W_{12}}  \int_{\RR.\bs{\xi}_{13}} \int_{W_{12}}h_1(\vv_1) \prod_{\ell=2}^3 h_\ell (\vw_{1\ell})\prod_{\ell=2}^3 \chi^{}_{1\ell} (T^2t_{1\ell})\\
&\hspace{0.2in} 
\times \chi^{}_{23} \left(T^2 \langle t_{12}\bs{\xi}_{12}+\vw_{12}, t_{13}\bs{\xi}_{13}+\vw_{13}\rangle \right)
\|\vv_1\|^2 (T^{2n})^3 (d\vw_{13} d\bs{\xi}_{13}) (d\vw_{12} d\bs{\xi}_{12}) d\vv_1+\sum_{\ell=1}^2\mathcal E_\ell.
\end{split}\]

\vspace{0.1in}
\noindent {\bf Step 4.} \quad
This step is similar to {\bf Step 1}, but this time we cut out a small ball in the $(2n-1)$-dimensional subspace $W_{12}=W_{12}(\vv_1)$ for each $\vv_1$. Set
\[\begin{split}
&\mathcal E_3=\int_{\RR^{2n}} \int_{\RR.\bs{\xi}_{12}} \int_{W_{12}\cap B_{\frac 1 T}(\origin)}
\int_{\RR.\bs{\xi}_{13}} \int_{W_{13}} \mathcal S_0 \prod_{\ell=2}^3 \chi^{}_{1\ell} (T^2t_{1\ell})\\
&\hspace{0.5in} \times \chi^{}_{23}\left(T^2\langle t_{12}\bs{\xi}_{12}+\vw_{12}, t_{13}\bs{\xi}_{13}+\vw_{13}\rangle \right)\|\vv_1\|^2 (T^{2n})^3 (d\vw_{12}d\bs{\xi}_{13}) (d\vw_{12} d\bs{\xi}_{12}) d\vv_1
\end{split}\]
so that for $T>T_0$,
\[\begin{split}
\eqref{main volume 1}&=
\int_{\RR^{2n}\setminus B_{\frac 1 T}(\origin)} \int_{\RR.\bs{\xi}_{12}} \int_{W_{12}\setminus B_{\frac 1 T}(\origin)}  \int_{\RR.\bs{\xi}_{13}} \int_{W_{12}}h_1(\vv_1) \prod_{\ell=2}^3 h_\ell (\vw_{1\ell})\prod_{\ell=2}^3 \chi^{}_{1\ell} (T^2t_{1\ell})\\
&\hspace{0.2in} 
\times \chi^{}_{23} \left(T^2 \langle t_{12}\bs{\xi}_{12}+\vw_{12}, t_{13}\bs{\xi}_{13}+\vw_{13}\rangle \right)
\|\vv_1\|^2 (T^{2n})^3 (d\vw_{13} d\bs{\xi}_{13}) (d\vw_{12} d\bs{\xi}_{12}) d\vv_1+\sum_{\ell=1}^3\mathcal E_\ell.
\end{split}\]

\vspace{0.1in}
\noindent {\bf Step 5.} \quad
We may assume that $(\vv_1, \vw_1)$ on the domain of the integral above is linearly independent. 
Define
\[
\bs{\xi}_{23}:=\frac {-B_0\vw_{12}}{\|\vw_{12}\|^2} \in W_{13}
\quad\text{and}\quad
W_{23}:=\ker \langle \vv_1, \cdot\rangle \cap \ker \langle\vw_{12}, \cdot\rangle \subseteq W_{13}.
\]
As before, $\bs{\xi}_{23}$ and $W_{23}$ are continuous maps $\bs{\xi}_{23}(\vv_1, \vw_1)$ and $W_{23}(\vv_1, \vw_1)$ from $$\{(\vv_1, \vw_{12})\in \RR^{2n}\times \RR^{2n}: \langle \vv_1, \vw_{12}\rangle=0\;\text{and}\; \vv_1, \vw_2\text{ are linearly independent}\}$$ to $\RR^{2n}$ and $\Gr_{2n-2}(\RR^{2n})$, respectively. 
It follows that any element $ \vw_{13}\in W_{13}$ can be decomposed as $\vw_{13}=t_{23}\bs{\xi}_{23}+\vw_{23}$, where $\vw_{23}\in W_{23}$ and $d\vw_{13}=\|\vw_{12}\|d\vw_{23}d\bs{\xi}_{23}$.
Using these coordinates, one can express \eqref{main volume 1} as follows.
\[\begin{split}
\eqref{main volume 1}&=
\int_{\RR^{2n}\setminus B_{\frac 1 T}(\origin)} \int_{\RR.\bs{\xi}_{12}} \int_{W_{12}\setminus B_{\frac 1 T}(\origin)}  \int_{\RR.\bs{\xi}_{13}} \int_{\RR.\bs{\xi}_{23}}\int_{W_{23}}h_1(\vv_1) h_2(\vw_{12}) h_3 (t_{23}\bs{\xi}_{23}+\vw_{23})\\
&\hspace{1in} 
\times \prod_{\ell=2}^3 \chi^{}_{1\ell} (T^2t_{1\ell})\chi^{}_{23} \left(T^2 \langle t_{12}\bs{\xi}_{12}+\vw_{12}, t_{13}\bs{\xi}_{13}+t_{23}\bs{\xi}_{23}+\vw_{23}\rangle \right)
 \\
&\hspace{2.1in}\times\|\vv_1\|^2\|\vw_{12}\| (T^{2n})^3(d\vw_{23}d\bs{\xi}_{23} d\bs{\xi}_{13}) (d\vw_{12} d\bs{\xi}_{12}) d\vv_1+\sum_{\ell=1}^3\mathcal E_\ell.
\end{split}\]

\vspace{0.1in}
\noindent {\bf Step 6.} \quad
We claim that there is $T_1\gg 1$ such that if $T>T_1$, it holds that
\begin{equation}\label{eq 1: prop volume formula}
h_3(t_{23}\bs{\xi}_{23}+\vw_{23})= h_3(\vw_{23}) + O\left( \frac {\mathcal S_1(h_3)} {\|\vw_{12}\|T^2}\right).
\end{equation}
To see this, as in {\bf Step 3}, we need to investigate the range of $t_{23}$.
Since $J_n$ is skew-symmetric and $\langle J_n\vx_1, J_n\vx_2\rangle=-\langle \vx_1, \vx_2 \rangle$, $\forall \vx_1, \vx_2\in \RR^{2n}$, it follows that
\[
\langle t_{12}\bs{\xi}_{12}+\vw_{12}, t_{13}\bs{\xi}_{13}+t_{23}\bs{\xi}_{23}+\vw_{23}\rangle
=t_{23} + t_{12}\langle \bs{\xi}_{12}, \vw_{23}\rangle+t_{13}\langle \vw_{12}, \bs{\xi}_{23}\rangle.
\]
Let $b=\max\{\langle \vx_1, \vx_2\rangle: \vx_1, \vx_2\in B_R(\origin)\}$. Since $t_{12}$ and $t_{23}$ are between $[-N/T^2, N/T^2]$, and $\supp\chi^{}_{23}$ is also contained in $[-N/T^2, N/T^2]$, it holds that
\[
t_{23} \in \left[ -\frac {(2b+1)N} {T^2}, \frac {(2b+1)N} {T^2}\right].
\]
In particular, there is $T_1\ge T_0$, depending continuously on $N$ and $R$,  so that \eqref{eq 1: prop volume formula} holds.
Put 
\[\begin{split}
\mathcal E_4&=
\int_{\RR^{2n}\setminus B_{\frac 1 T}(\origin)} \int_{\RR.\bs{\xi}_{12}} \int_{W_{12}\setminus B_{\frac 1 T}(\origin)}  \int_{\RR.\bs{\xi}_{13}} \int_{\RR.\bs{\xi}_{23}}\int_{W_{23}} \mathcal S_1\prod_{\ell=2}^3 \chi^{}_{1\ell} (T^2t_{1\ell})\\
&\hspace{0.2in} 
\times \chi^{}_{23} \left(T^2 (t_{23}+t_{12}\langle \bs{\xi}_{12},\vw_{23}\rangle +t_{13}\langle \vw_{12}, \bs{\xi}_{23}\rangle\right)\|\vv_1\|^2 T^{6n-2}(d\vw_{23}d\bs{\xi}_{23} d\bs{\xi}_{13}) (d\vw_{12} d\bs{\xi}_{12}) d\vv_1.
\end{split}\]
Hence we obtain the equation
\[\begin{split}
\eqref{main volume 1}&=
\int_{\RR^{2n}\setminus B_{\frac 1 T}(\origin)} \int_{\RR.\bs{\xi}_{12}} \int_{W_{12}\setminus B_{\frac 1 T}(\origin)}  \int_{\RR.\bs{\xi}_{13}} \int_{\RR.\bs{\xi}_{23}}\int_{W_{23}}h_1(\vv_1) h_2(\vw_{12}) h_3 (\vw_{23})\\
&\hspace{1in} 
\times \prod_{\ell=2}^3 \chi^{}_{1\ell} (T^2t_{1\ell})\chi^{}_{23} \left(T^2 (t_{23}+t_{12}\langle \bs{\xi}_{12},\vw_{23}\rangle +t_{13}\langle \vw_{12}, \bs{\xi}_{23}\rangle\right)
 \\
&\hspace{2.1in}\times\|\vv_1\|^2\|\vw_{12}\| (T^{2n})^3(d\vw_{23}d\bs{\xi}_{23} d\bs{\xi}_{13}) (d\vw_{12} d\bs{\xi}_{12}) d\vv_1+\sum_{\ell=1}^4\mathcal E_\ell.
\end{split}\]

\vspace{0.1in}
\noindent {\bf Step 7.}\quad
For given $\vv_1, \;\vw_{12},\;t_{12}\bs{\xi}_{12}$ and  $t_{23}\bs{\xi}_{23}$, it holds that
\[
\int_{\RR.\bs{\xi}_{23}} \chi^{}_{23}\left(T^2 (t_{23}+t_{12}B_0(\bs{\xi}_{12},\vw_{23})+t_{13}B_0(\vw_{12}, \bs{\xi}_{23})\right)d\bs{\xi}_{23}=\frac 1 {T^2} (b_{23}-a_{23}).
\]
Likewise, one can compute inner integrals over $\RR.\bs{\xi}_{12}$ and $\RR.\bs{\xi}_{13}$ consecutively (under the circumstance where $\vv_1$ is given). It follows that
\begin{equation}\label{eq 2:prop volume formula}\begin{split}
\eqref{main volume 1}&=
\int_{\RR^{2n}\setminus B_{\frac 1 T}(\origin)} \int_{W_{12}\setminus B_{\frac 1 T}(\origin)}  \int_{W_{23}}h_1(\vv_1) h_2(\vw_{12}) h_3 (\vw_{23})\|\vv_1\|^2\|\vw_{12}\| d\vw_{23} d\vw_{12} d\vv_1\\
&\hspace{1.5in}\times\prod_{1\le i<j\le 3} (b_{ij}-a_{ij})\cdot T^{6n-6}+\sum_{\ell=1}^4\mathcal E_\ell.
\end{split}\end{equation}

Now, we handle the error terms $\mathcal E_\ell$ for $1\le \ell\le 4$.
Let us apply {\bf Step 2} and {\bf Step 5} (change of variables), and compute inner integrals over $\RR.\bs{\xi}_{23}$ and consecutively over $\RR.\bs{\xi}_{12}$ and $\RR.\bs{\xi}_{13}$ in each $\mathcal E_{\ell}$. 
Since $\supp h_{\ell}\subseteq B_R(\origin)$ for all $1\le\ell\le 3$, one can deduce that
\[
\sum_{\ell=1}^4 \mathcal E_{\ell}=
O\left(\mathcal S_0 T^{(6n-6)-(2n-1)}\right)+O\left(\mathcal S_1 T^{(6n-6)-2}\right)+O\left(\mathcal S_2 T^{(6n-6)-4}\right)
\]
for $T\ge T_1$, where implicit constants of error terms depend on $R$ and $N$, continuously.

\vspace{0.1in}
\noindent {\bf Step 8.}\quad
Finally, we want to recover the first line of (R.H.S) in \eqref{eq 2:prop volume formula} to
\[
J(h_1,\ldots, h_\ell)=\int_{\RR^{2n}}\int_{W_{12}}\int_{W_{23}} h_1(\vv_1)h_2(\vw_{12})h_3(\vw_{23})
\|\vv_1\|^2 \|\vw_{12}\| d\vw_{23}d\vw_{12}d\vv_1.
\]  
Notice that the integral domain above coincides exactly with the cone $\mathcal C_{I,k}$ of the symplectic form $\langle \cdot, \cdot\rangle$ in $(\RR^{2n})^k$.
This procedure is similar to repeating {\bf Step 4} and {\bf Step 1} but backward and the difference is bounded by $O\left(\mathcal S_0T^{6n-6-(2n-1)}\right)$. Therefore, we obtain the theorem for the case when $k=3$.

\vspace{0.1in}
Similar to the rank-3 case, the key idea of the proof of general cases is to find the coordinate system of $(\RR^{2n})^k$ which are able to parametrize the values of $\langle \vv_i, \vv_j\rangle$, $i<j\le k$, for given $\vv_1, \ldots, \vv_i\in \RR^{2n}$, up to translations; and estimate values of $\prod_{\ell=1}^k h_\ell$ at points on the support to function values at nearby points on the cone $\mathcal C_{B_0, k}$.
\end{proof}

We remark that in the proof above, we needed to fix the standard symplectic form $\langle \cdot, \cdot \rangle$ for using the fact that the corresponding skew-symmetric matrix $J_n$ is orthogonal in {\bf Step 6}. For the case when $k=2$, the process does not proceed until this step, and one can replace $\langle \cdot, \cdot \rangle$ by any symplectic forms.
 
\begin{proof}[Proof of \Cref{thm: volume formula}.]
Let $h_0=\chi^{}_{B_1(\mathbf 0)}$ be the characteristic function of the unit ball. For each $\delta>0$, small enough, choose smooth functions $h^{\pm}_\delta:\RR^{2n}\rightarrow [0,1]$ which approximate $h_0$ above and below, i.e.,
\[
h^{-}_{\delta}(\vv)=\left\{\begin{array}{cl}
1, &\text{if }\|\vv\|\le 1-\delta;\\
0, &\text{if }\|\vv\|\ge 1
\end{array}\right.
\quad\text{and}\quad
h^{+}_{\delta}(\vv)=\left\{\begin{array}{cl}
1, &\text{if }\|\vv\|\le 1;\\
0, &\text{if }\|\vv\|\ge 1+\delta.
\end{array}\right.
\]
One can further assume that $\mathcal S_1(h^{\pm}_\delta)=O(1/\delta)$.

For a given $g\in G_{2n}$, let $h_{0,g}(\vv)=h_0(g^{-1}\vv)$ and $h^{\pm}_{\delta, g}(\vv)=h^{\pm}_{\delta}(g^{-1}\vv)$.

Then the volume that we want to estimate is 
\begin{equation}\label{eqn 1: volume}
\int_{\RR^{2n}} \cdots \int_{\RR^{2n}} \prod_{\ell=1}^k h_{0,g}\left(\frac {\vv_\ell} T\right) 
\prod_{1\le i<j\le k} \chi_{ij}(\langle \vv_i, \vv_j\rangle ) d\vv_k \cdots d\vv_1,
\end{equation}
where $\chi_{ij}$ is the characteristic function of the interval $(a_{ij}, b_{ij})\subseteq \RR$ for $1\le i<j\le k$.
It is obvious that \eqref{eqn 1: volume} is bounded above and below by 
\[
\int_{\RR^{2n}} \cdots \int_{\RR^{2n}} \prod_{\ell=1}^k h^{\pm}_{\delta,g}\left(\frac {\vv_\ell} T\right) 
\prod_{1\le i<j\le k} \chi_{ij}(\langle \vv_i, \vv_j\rangle ) d\vv_k \cdots d\vv_1
\]
that one can apply Proposition~\ref{prop: volume formula}. It follows that
\[\begin{split}
&\int_{\RR^{2n}} \cdots \int_{\RR^{2n}} \prod_{\ell=1}^k h^{\pm}_{\delta,g}\left(\frac {\vv_\ell} T\right) 
\prod_{1\le i<j\le k} \chi_{ij}(\langle \vv_i, \vv_j\rangle ) d\vv_k \cdots d\vv_1\\
&=J(h^{\pm}_{\delta, g})\prod_{1\le i<j\le k} (b_{ij}-a_{ij})\cdot T^{2nk-k(k-1)} 
+O\left(\mathcal S_0 T^{2nk-k(k-1)-(2n-k-1)}\right)+O\left(\mathcal S_t T^{2nk-k(k-1)-2t}\right),
\end{split}\]
where
\[
J(h^{\pm}_{\delta,g})=J(h^{\pm}_{\delta, g}, \ldots, h^{\pm}_{\delta, g})=\int_{\mathcal C_{I,k}} \prod_{\ell=1}^k h^{\pm}_{\delta, g}(\vw_\ell) \prod_{\ell=1}^{k-1} \|\vw_\ell\|^{k-\ell} d\vw_k \cdots d\vw_1.
\]

It holds that
\[\begin{split}
&\left|J(h^{\pm}_{\delta, g}) - J(h_{0, g})\right|\le J(h^+_{\delta, g})- J(h^-_{\delta, g})\\
&\quad=\int_{C_{I,k}}\left(\sum_{i=1}^k \prod_{\ell=1}^{i-1} h^+_{\delta, g} (\vw_\ell) \left(h^+_{\delta, g}(\vw_i) - h^-_{\delta, g}(\vw_i)\right)\prod_{\ell=i+1}^k h^-_{\delta,g}(\vw_\ell)\right)\prod_{\ell=1}^{k-1} \|\vw_{\ell}\|^{k-\ell} d\vw_k \cdots d\vw_1\\
&\quad=O_g\left(\delta\right)
\end{split}\]
since for each $i=1,\ldots, k$, 
\[\begin{split}
&\int_{C_{I,k}} \prod_{\ell=1}^{i-1} h^+_{\delta, g} (\vw_\ell) \left(h^+_{\delta, g}(\vw_i) - h^-_{\delta, g}(\vw_i)\right)\prod_{\ell=i+1}^k h^-_{\delta,g}(\vw_\ell)\prod_{\ell=1}^{k-1} \|\vw_{\ell}\|^{k-\ell} d\vw_k \cdots d\vw_1\\
&\quad \le (2\|g\|_{\infty})^{k(k-1)/2} \cdot \prod_{1\le \ell\neq i\le k} \left(G_{\ell} V_{2n-(\ell-1)}\right)\cdot G_i\left((1+\delta)^{2n-(i-1)}-(1-\delta)^{2n-(i-1)}\right)V_{2n-(i-1)},
\end{split}\]
where $\|g\|_{\infty}$ is the operator norm of $g$ on $\RR^{2n}$, $G_i$ is the $(2n-(i-1))$-power of the operator norm of $g$ on the Grassmannian space $\Gr_{2n-(i-1))}(\RR^2n)$, and $V_{2n-(\ell-1)}$ is the $(2n-(\ell-1))$-dimensional volume of the unit ball.

Thus, the volume 
\[\begin{split}
\eqref{eqn 1: volume}&=J(h_{0,g})\prod_{1\le i<j\le k} (b_{ij}-a_{ij})\cdot T^{2nk-k(k-1)}
+O\left(\delta T^{2nk-k(k-1)}\right)\\
&+O\left(\delta^{-1}T^{2nk-k(k-1)-(2n-k-1)}\right)+\sum_{t=1}^{k-1} O\left(\delta^{-t}T^{2nk-k(k-1)-2t}\right).
\end{split}\]
The result follows when we take $\delta=T^{-1}$ for $2\le k\le 2n-3$ and $\delta=T^{-1/2}$ for $k=2n-2$.
\end{proof}

\section{Use of Rogers' Formulas and Proof of Theorem~\ref{thm: quantitative}}\label{Sec:Use of Rogers' Formulas}

\begin{definition}\label{def: Siegel transform}
For a bounded and compactly supported function $F:(\RR^d)^k\rightarrow \RR$, define
\[
\widetilde{F}(g\Gamma_{d})=\widetilde{F}(g\ZZ^d)=\sum_{\vv_i\in \ZZ^d} F(g\vv_1, \ldots, g\vv_k).
\]
The function $\widetilde F$ is defined on $G_d/\Gamma_d=\SL_d(\RR)/\SL_d(\ZZ)$.
\end{definition}

In \cite{Rogers55}, Rogers introduced the integral formula below (see also \cite{Schmidt57}).
\begin{theorem}\label{thm: Rogers}
\[\begin{split}
\int_{G_d/\Gamma_d} \widetilde{F}(g\ZZ^d)
&=F(\origin, \ldots, \origin) + \int_{(\RR^d)^k} F(\vv_1, \ldots, \vv_k) d\vv_1 \ldots \vv_k\\
&+\sum_{r=1}^{k-1} \sum_{q\in \NN} \sum_{D\in \mathcal D^k_{r,q}} c_D\int_{(\RR^d)^r} F\left((\vv_1, \ldots, \vv_r) \frac D q \right) d\vv_1 \cdots d\vv_r.
\end{split}\]
Here, $\mathcal D^k_{r,q}$ is the set of $r\times k$ integer matrices $D$ such that $D/q$ is the reduced row-echelon matrix of rank $r$.
The integral weight $c_D$ for each $D\in \mathcal D^k_{r,q}$ is the $d$-th power of the reciprocal of the covolume of the lattice $\left(D/q\right)^{-1}\hspace{-0.03in}(\ZZ^k)\cap \ZZ^r$ in $\RR^r$, where we regard the $r\times k$ matrix $D/ q$ as a map from $\RR^r$ to $\RR^k$.
\end{theorem}

For a bounded and compactly supported function $F:(\RR^{d})^k\rightarrow \RR$, consider the discrepancy function on $G_{d}/\Gamma_{d}$ defined as 
\[
D=D\left(g\Gamma_{d}, F_i\right)=D\left(g\ZZ^{d}, F_i\right)=\widetilde{F_i}(g\ZZ^{d}) - \EE(\widetilde{F_i}),
\]
where $\EE(\widetilde{F_i})$ is the average of $\widetilde{F_i}$ over $G_{d}/\Gamma_{d}$ with respect to the $G_{d}$-invariant probability measure $\mu$.
When $k=1$ and $F$ is the characteristic function of a Borel set $A\subseteq \RR^{d}$, $\EE(\widetilde{F_i}=\vol(A)$ thus in this case, the definition above is identical to the discrepancy function defined in \cite{KY2020}. The following lemma is easy to deduce.

\begin{lemma}\label{lem: discrepancy}
Let $F_i:(\RR^{d})^k\rightarrow \RR$, $i=1,2,3$, be bounded and compactly supported functions such that $F_1\le F_2 \le F_3$. 
It holds that
\[
D\left(g\ZZ^{d}, F_2\right)\le \max\left\{D\left(g\ZZ^{d}, F_1\right)D\left(g\ZZ^{d}, F_3\right)\right\}.
\]
\end{lemma}

We also need the theorem below, which extends \cite[Theorem 2.4]{AM2009} to the higher-rank setting. The proof mainly adapts the arguments in \cite[Section 9]{Rogers55B}, and for the reader's convenience we present the complete details.
\begin{theorem}\label{prop: upper bound}
Assume that $2n\ge \max\{r(2k-r)+3: 1\le r \le 2k-1\}$ (i.e., $2n\ge k^2+3$).
Let $E\subseteq (\RR^{d})^k$ be a Jordan-measurable set such that there is a collection $\{E_r\}_{1\le r\le k}$ of Jordan-measurable sets $E_r\subseteq (\RR^{d})^r$ satisfying the following property: For any projection $\pi_r: (\RR^{d})^k\rightarrow (\RR^{d})^r$ of the form
\[
(\vv_1, \ldots, \vv_k) \mapsto (\vv_{j_1}, \ldots, \vv_{j_r}),\;\text{where}\; 1\le j_1<\cdots<j_r\le k,
\] 
it holds that
\begin{equation}\label{eqn1: upper bound}
\pi_r(E)\subseteq E_r.
\end{equation}
We conventionally assume that $E_k=E$ and $\pi_k$ is the identity map. There is a constant $C>0$, depending only on the dimension $d$ so that
\[\begin{split}
0\le \int_{G_{d}/\Gamma_{d}} \widetilde{\chi^{}_E} (g\ZZ^{d})^2 d\mu(g)
&-\left(\int_{G_{d}/\Gamma_{d}} \widetilde{\chi^{}_E} (g\ZZ^{d}) d\mu(g)\right)^2\\
&\le C \max\left\{\vol_{r_1}(E_{r_1})\cdot \vol_{r_2}(E_{r_2}) : 1\le r_1+r_2\le 2k-1\right\}.
\end{split}\]
Here, we allow the case when $r_1=0$ (or $r_2=0$) and in this case, $\vol_{r_1}(E_{r_1})\cdot \vol_{r_2}(E_{r_2})=\vol_{r_2}(E_{r_2})$.
\end{theorem}
\begin{proof}
Using the $k$-th and $2k$-th formulas of Rogers (\Cref{thm: Rogers}), one can show the following.
\begin{align}
&\int_{G_{d}/\Gamma_{d}} \widetilde{\chi^{}_E} (g\ZZ^{d})^2 d\mu(g)
-\left(\int_{G_{d}/\Gamma_{d}} \widetilde{\chi^{}_E} (g\ZZ^{d}) d\mu(g)\right)^2\label{eqn2: upper bound}\\
&\hspace{1in}= \sum_{r=1}^{2k-1} \sum_{q\in \NN} \sum_{D\in \mathcal D'^{2k}_{r,q}} 
c_D\int_{(\RR^{d})^r} \chi^{}_E\otimes \chi^{}_E\left( (\vv_1, \ldots, \vv_r)\frac D q \right) d\vv_1 \cdots d\vv_r, \nonumber
\end{align}
where $\mathcal D'^{2k}_{r,q}$ is the set of $r\times 2k$ matrices $D\in \mathcal D^{2k}_{r,q}$ which is not of the following forms:
\begin{itemize}
\item $D$ is $\left(D_0| O\right)$ or $\left(O|D_0\right)$ for some $D_0\in \mathcal D^{k}_{r,q}$ and $O$ is the zero matrix of size $r\times k$;
\item there are $D_1\in \mathcal D^k_{r_1, q_1}$ and $D_2\in \mathcal D^k_{r_2, q_2}$, where $r_1+r_2=r$ and $\lcm(q_1,q_2)=q$ so that
\[
D=\left(\begin{array}{cc}
\frac q {q_1}D_1 & \\
& \frac q {q_2}D_2\end{array}\right).
\] 
\end{itemize}
 Indeed, it can be easily seen that all integral terms of $\left(\int_{G_{d}/\Gamma_{d}} \widetilde{\chi^{}_E} (g\ZZ^{d}) d\mu(g)\right)^2$ are canceled by the terms in $\int_{G_{d}/\Gamma_{d}} \widetilde{\chi^{}_E} (g\ZZ^{d})^2 d\mu(g)$ that correspond to the matrices of the form in the list above (when $r=k$, we will consider $\mathcal D^k_{k,1}=\{\Id_k\}$).
For example, the term
\[
2c_{D_1}\int_{(\RR^{d})^{r_1}} \chi^{}_E\left((\vv_1, \ldots, \vv_{r_1})\frac {D_1} {q_1}\right)d\vv_1\cdots d\vv_{r_1}
\cdot c_{D_2}\int_{(\RR^{d})^{r_2}} \chi^{}_E\left((\vv_1, \ldots, \vv_{r_2})\frac {D_2} {q_2}\right)d\vv_1\cdots d\vv_{r_2},
\]
where $D_1\in \mathcal D^k_{r_1, q_1}$ and $D_2\in \mathcal D^k_{r_2, q_2}$, is canceled by the sum of two integrals
\[\begin{gathered}
c_{D_3}\int_{(\RR^{d})^{r_1+r_2}}\chi^{}_E\otimes \chi^{}_E\left(\vv_1, \ldots, \vv_r)\frac {D_3} q\right) d\vv_1 \cdots d\vv_r\\
+c_{D_4}\int_{(\RR^{d})^{r_1+r_2}}\chi^{}_E\otimes \chi^{}_E\left(\vv_1, \ldots, \vv_r)\frac {D_4} q\right) d\vv_1 \cdots d\vv_r,
\end{gathered}\]
where $D_3=\left(\begin{array}{cc}
\frac q {q_1}D_1 & \\
& \frac q {q_2}D_2\end{array}\right)$ and $D_4=\left(\begin{array}{cc}
\frac q {q_2}D_2 & \\
& \frac q {q_1}D_1\end{array}\right)$ are elements of $\mathcal D^{2k}_{r,q}$ with $q=\lcm(q_1, q_2)$. It follows from the definition that
\[
c_{D_3}=c_{D_4}=c_{D_1}\cdot c_{D_2}.
\]

\vspace{0.1in}
Denote
\[
V= \max\left\{\vol_{r_1}(E_{r_1})\cdot \vol_{r_2}(E_{r_2}) : 1\le r_1+r_2\le 2k-1\right\}.
\]

Let an $r\times 2k$ matrix $D\in \mathcal D^{2k}_{r,q}$ be given. By the change of coordinates, we may assume that
\begin{equation}\label{eqn3: upper bound}
\frac D q=\left(\Id_r | \:C \right)
\end{equation}
for some $r\times (2k-r)$ matrix $C=(c_{ij})_{ij}$, with indexing $1\le i\le r$ and $r+1\le j\le 2k$.
Let us show that
\begin{equation}\label{eqn4: upper bound}
\int_{(\RR^{d})^r} \chi^{}_E\otimes \chi^{}_E \left((\vv_1, \ldots, \vv_r)\frac D q\right) d\vv_1 \cdots d\vv_r
\le \min\left\{1, \frac 1 {c(D)^{d}}\right\}V,
\end{equation}
where $c(D)=\max\{c_{ij}\}$. 

Let $A$ be any $r\times r$ minor of $D/q$ with $\det A\neq 0$. Take $r_1\ge 0$ such that
\[
1\le j_1 < \cdots <j_{r_1} \le k <k+1\le j_{r_1+1} < \cdots <j_r \le 2k,
\]
where $j_1 < \cdots < j_r$ are indices of columns of $D$ consisting of $A$ and $r_2=r-r_1$. Set the coordinate projection
\[\begin{gathered}
\pi_{r_1}: (\vv_1, \ldots, \vv_k)\in (\RR^{d})^k \mapsto (\vv_{j_1}, \ldots, \vv_{j_{r_1}})\in (\RR^{d})^{r_1};\\
\pi_{r_2}: (\vv_1, \ldots, \vv_k)\in (\RR^{d})^k \mapsto (\vv_{j_{r_1+1}-k}, \ldots, \vv_{j_r-k})\in (\RR^{d})^{r_2}.
\end{gathered}\] 
If $r_1$ or $r_2=0$, then we do not think the corresponding projection. Our assumption tells us that
\[
\chi^{}_E\otimes \chi^{}_E \le (\chi^{}_{E_{r_1}}\otimes\chi^{}_{E_{r_2}})\circ (\pi_{r_1} \otimes \pi_{r_2})
\]
(again, if one of $r_1, r_2$ is zero, saying $r_1=0$, then the above inequality is $\chi^{}_{E}\le \chi^{}_{E_{r_2}}\circ \pi_{r_2}$. We will skip this case from now on since it is easily covered by the case when $r_1r_2\neq 0$).
Thus
\[\begin{split}
\text{(L.H.S) of }\eqref{eqn3: upper bound}
&\le \int_{(\RR^{d})^r} \chi^{}_{E_{r_1}}\otimes \chi^{}_{E_{r_2}} \left((\vv_1, \ldots, \vv_r)A\right) d\vv_1 \cdots d\vv_r\\
&= \int_{(\RR^{d})^r} \chi^{}_{E_{r_1}}\otimes \chi^{}_{E_{r_2}} \left((\vv_1, \ldots, \vv_r)\right) |\det A|^{-2n}d\vv_1 \cdots d\vv_r
\le |\det A|^{-2n} V.
\end{split}\]
Moreover, since $A$ is any minor of $D/q$ with $\det A\neq 0$, it follows that
\[
\text{(L.H.S) of }\eqref{eqn3: upper bound}\le \max\{|\det A|: A\text{ is minor of } D/q\}^{-2n}V
\]
which is bounded above by $V$ if we consider $A=\Id_r$, and by any $|c_{ij}|^{-2n}V$ if we consider $A$ consisting all the columns of $\Id_r$ except the $i$-th and the $j$-th column of $D/q$. This shows the inequality \eqref{eqn3: upper bound}.

\vspace{0.1in}
Now, let us prove the theorem. Denote
\[
N=\max\left\{\frac {2k!} {r! (2k-r)!}: 1\le r \le 2k-1 \right\}.
\]
Note that
\begin{equation}\label{eqn5: upper bound}
c_D\le \frac 1 {q^{d}}
\end{equation}
for all $D\in \mathcal D^{2k}_{r,q}$ (see \cite[Equations (4) and (46)]{Rogers55B} in the notation there, $c_D=(N(C)/q^r)^{d}$).

We divide the summation of integrals in \eqref{eqn2: upper bound} with respect to $D\in \mathcal D'^k_{r,q}$ in three cases, for a given $r=1, \ldots, 2k-1$, and get an upper bound in each cases.

\vspace{0.05in}
\noindent\textbf{Case I.} $q\ge 2$ and $c(D)\le 1$.

The number of such $D\in \mathcal D'^{2k}_{r,q}$ is bounded by 
\[
\frac {2k!} {r!(2k-r)!} (2q+1)^{r(2k-r)}\le N \left(\frac 5 2 q\right)^{r(2k-r)}.
\]

Since $-2n+r(2k-r)+2\le 0$, and by \eqref{eqn4: upper bound} and \eqref{eqn5: upper bound},
\[\begin{split}
&\sum_{q\in \NN_{\ge 2}} \sum_{D\in \mathcal D'^{2k}_{r,q}} c_D \int_{(\RR^{d})^r} \chi^{}_E\otimes \chi^{}_E\left((\vv_1, \ldots, \vv_r)\frac D q \right)d\vv_1\cdots d\vv_r \le \sum_{q\in \NN_{\ge 2}} N\left(\frac 5 2 q\right)^{m(2k-r)} q^{-2n} V\\
&\hspace{1in}\le N \left( \frac 5 2\right)^{r(2k-r)}2^{-2n+r(2k-r)+2} V\sum_{q\in \NN_{\ge2}} q^{-2}
\le N 5^{r(2k-r)} 2^{-2n+2}\: V.
\end{split}\]

\vspace{0.05in}
\noindent\textbf{Case II.} $q\ge 1$ and $c(D)>1$.

For each $q\ge 1$ and $\ell\in \NN_{>q}$, the number of $D\in \mathcal D'^{2k}_{r,q}$ with $c(D)=\ell/q$ is bounded by
\[
N(2\ell+1)^{r(2k-r)}\le N \left( \frac 5 2 \ell\right)^{r(2k-r)}.
\]
By \eqref{eqn4: upper bound} and \eqref{eqn5: upper bound}, it holds that
\[
\sum_{\scriptsize \begin{array}{c}
D\in \mathcal D'^{2k}_{r,q}\\
c(D)=\ell/q\end{array}}  c_D \int_{(\RR^{d})^r} \chi^{}_E\otimes \chi^{}_E\left((\vv_1, \ldots, \vv_r)\frac D q \right)d\vv_1\cdots d\vv_r
\le N \left(\frac 5 2\right)^{r(2k-r)} q^{-2n} c^{-2n} V.
\]
Hence using the fact that $-2n+r(2k-r)+2\le 0$ again,
\[\begin{split}
&\sum_{\ell\in \NN_{>q}} \sum_{\scriptsize \begin{array}{c}
D\in \mathcal D'^{2k}_{r,q}\\
c(D)=\ell/q\end{array}}  c_D \int_{(\RR^{d})^r} \chi^{}_E\otimes \chi^{}_E\left((\vv_1, \ldots, \vv_r)\frac D q \right)d\vv_1\cdots d\vv_r
\le \left( \frac 5 2\right)^{r(2k-r)} V \sum_{\ell\in \NN_{>q}} \ell^{-2n+r(2k-r)}\\
&\hspace{0.8in}\le \left(\frac 5 2 \right)^{r(2k-r)} V (q+1)^{-2n+r(2k-r)+2}\sum_{\ell\in \NN_{>q}} \frac 1 {\ell^2}\
\le 2N\left(\frac 5 2 \right)^{r(2k-r)} V (q+1)^{-2n+r(2k-r)+1},
\end{split}\]
where in the last inequality, we use the inequality
\[
(q+1)\sum_{\ell\in \NN_{>q}} \frac 1 {\ell^2} \le (q+1)\sum_{\ell\in \NN_{>q}} \frac 1 {\ell(\ell-1)}=\frac {q+1} {q} \le 2.
\]

One can conclude that, using the fact that $-2n+r(2k-r)+3\le 0$ in this time, it follows that
\[\begin{split}
&\sum_{q\in \NN} \sum_{\ell\in \NN_{>q}}\sum_{\scriptsize \begin{array}{c}
D\in \mathcal D'^{2k}_{r,q}\\
c(D)=\ell/q\end{array}}  c_D \int_{(\RR^{d})^r} \chi^{}_E\otimes \chi^{}_E\left((\vv_1, \ldots, \vv_r)\frac D q \right)d\vv_1\cdots d\vv_r\\
&\hspace{0.6in}\le 2N \left(\frac 5 2 \right)^{r(2k-r)} V \sum_{q\in \NN} (q+1)^{-2n+r(2k-r)+1}\\
&\hspace{0.6in}\le 2N \left(\frac 5 2 \right)^{r(2k-r)} V 2^{-2n+r(2k-r)+3} \sum_{q=1} (q+1)^{-2}\\
&\hspace{0.6in} \le N5^{r(2k-r)}2^{-2n+4} \:V.
\end{split}\]

\vspace{0.05in}
\noindent\textbf{Case III.} $q=1$ and $c(D)=1$ 

In this case, possible $D\in \mathcal D'^{2k}_{r,q}$ is a matrix consisting of 0 or $\pm 1$ only. Thus,
\[
\sum_{\scriptsize \begin{array}{c}
D\in \mathcal D'^{2k}_{r,q}\\
c(D)=1\end{array}}\int_{(\RR^{d})^r} \chi^{}_E\otimes \chi^{}_E\left((\vv_1, \ldots, \vv_r)\frac D q \right) d\vv_1 \cdots d\vv_r
\le N 3^{r(2k-r)}V.
\] 

Therefore, the theorem holds if we take a constant $C$ by
\[
C=N\sum_{r=1}^{2k-1} \left(20\cdot 5^{r(2k-r)}2^{-2n}+3^{r(2k-r)}\right).
\]

\end{proof}

\begin{remark}
We note that Rogers' higher moment formulas and their upper bounds \cite[Section 9]{Rogers55B} (and \cite[Lemma 7]{Rogers56}) also play central roles in the later works \cite{Rogers56, Sodergren2011, StSo2019, Kim2020}, but in a different regime. In those papers, the ambient dimension $d$ is allowed to grow, and one studies the distribution of the number of lattice points in Borel sets with certain volume(s), so a dimension-dependent negative-power term must be tracked. 
Here we keep the dimension fixed at $d=2n$ and letting the volume of the Borel sets diverge to infinity, thus eliminating the need for that negative-power factor (Especially, Case III in the proof above).
\end{remark}

The following corollary is crucial in applying the Borel--Cantelli lemma to derive the main theorem. Note that by adopting the collection $\{E_r\}$ in the previous theorem, one can maintain the lower bound of the dimension as $2n\ge k^2+3$. To demonstrate this, for example, suppose that we apply the results of \cite[Section 9]{Rogers55B} directly, by taking $\{E_r=(\chi^{}_{A})^r: 1\le r \le k\}$ for some $A\subseteq \RR^{2n}$ satisfying \eqref{eqn1: upper bound}. In this case, $A$ should be $h^{-1}B_T(\origin)$, and the lower bound of the dimension is changed to $2n\ge 2k(k-1)+1$ for the corollary work.

\begin{corollary}\label{cor: upper bound}

Assume $2n\ge k^2+3$. Let $\{(a_{ij}, b_{ij}): 1\le i<j\le k\}$ be a collection of bounded intervals in $\RR$. For any $g\in G_{2n}$, set
\[
E_{g,T}=\left\{(\vv_1, \ldots, \vv_k)\in (\RR^{2n})^k: \langle \vv_i, \vv_j \rangle \in (a_{ij}, b_{ij}) \right\}
\cap (g^{-1}B_T(\origin))^k.
\]
Then 
\begin{equation}\label{eqn1: cor upper bound}
\int_{G_{2n}/\Gamma_{2n}} \widetilde{\chi^{}_{E_{h,T}}} (g\ZZ^{2n})^2 d\mu(g)
-\left(\int_{G_{2n}/\Gamma_{2n}} \widetilde{\chi^{}_{E_{h,T}}} (g\ZZ^{2n}) d\mu(g)\right)^2
= O_g\left(T^{2n(2k-1)-2k^2}\right).
\end{equation}
\end{corollary}
\begin{proof}
Let $(a,b)$ is an interval such that $(a_{ij}, b_{ij})\subseteq (a,b)$ for all $1\le i<j\le k$.
If we take 
\[
E_r=\left\{(\vv_1, \ldots, \vv_r)\in (\RR^{2n})^r: \langle \vv_i, \vv_j\rangle \in (a,b)\right\}\cap (g^{-1}B_T(\origin))^r
\]
for $1\le r\le k-1$ and $E_k=E_{h,T}$, then the collection $\{E_r\}_{1\le r\le k}$ satisfies the condition on \Cref{prop: upper bound}.

From the \Cref{thm: volume formula}, for $1\le r\le k$,
\[
\vol_{2nr}(E_r)=O_h(T^{2nr-r(r-1)}).
\]
Thus by \Cref{prop: upper bound}, (L.H.S) of \eqref{eqn1: cor upper bound} is
\[
O\left(\max\left\{T^{2nr_1-r_1(r_1-1)+2nr_2-r_2(r_2-1)}: 1\le r_1+r_2\le 2k-1,\; 0\le r_1, r_2\le k\right\}\right).
\]
The maximum is obtained when $r_1+r_2=2k-1$ and $\{r_1, r_2\}$ is $\{k, k-1\}$, which is $(2n+1)(2k-1)-(2k-1)^2+2k(k-1)=2n(2k-1)-2k^2$. Under the assumption that $2n\ge k^2+3$, this is properly less than $2(2nk-k(k-1))$.
\end{proof}

One can refine the proof of Theorem~\ref{prop: upper bound} to obtain the following proposition, thus we skip the proof. The proposition tells us that if we take $F$ as the characteristic function of $E_{g,T}$, the average of its Siegel transform is no longer the volume $\vol_{2nk}(E_{g,T})$ when $k\ge 2$, but as $T\rightarrow \infty$, it converges asymptotically to $\vol_{2nk}(E_{g,T})$ with power-saving error term.

\begin{proposition}\label{prop: mean} Let $2n\ge (k/2)^2+3$. Let $E_{g,T}$ for $g\in G_{2n}$ be as in Corollary~\ref{cor: upper bound}. It holds that
\[
\int_{G_{2n}/\Gamma_{2n}} \widetilde{\chi^{}_{E_{g,T}}} (g\ZZ^{2n}) d\mu(g)=\vol_{2nk}(E_{g,T})+O_g\left(T^{2n(k-1)-(k-1)(k-2)}\right).
\]
\end{proposition}

\begin{proof}[Proof of \Cref{thm: quantitative}.]
It suffices to show that the theorem holds for almost all $g\in \mathcal K$ for any compact set $\mathcal K\subseteq G_{2n}$. Furthermore, since counting function is $\Gamma_{2n}$-invariant, we may assume that $\mathcal K$ is contained in the closure of a fundamental domain for $G_{2n}/\Gamma_{2n}$.

For each $g\in \mathcal K$, recall that
\[
E_{g,T}=\left\{(\vv_1, \ldots, \vv_k)\in (\RR^{2n})^k: \langle \vv_i, \vv_j \rangle \in (a_{ij}, b_{ij}) \right\}
\cap (gB_T(\origin))^k.
\]
The counting function in the theorem is the function value of the Siegel transform of the characteristic function of $E_{g,T}$ at $g\ZZ^{2n}$:
$$N_{g, \mathcal I}(T)=\widetilde{\chi^{}_{E_{g,T}}}(g\ZZ^{2n}).$$
Define
\[
\mathcal B(\mathcal K, \delta,T)=\left\{g\in \mathcal K: D(g\ZZ^{2n}, E_{g,T}) \ge \vol_{2nk}(E_{g,T})^{\delta}\right\},
\]
where $0<\delta<1$. We want to find (the range of) $\delta$, not depending on $\mathcal K$, for which the set 
$$\limsup_{T\rightarrow \infty} \mathcal B(\mathcal K, \delta, T)$$ 
is null.
Let $\alpha>1$ be some number to be determined later and consider the sequence $(T_\ell=\ell^\alpha)_{\ell\in \NN}$. We will apply the Borel--Cantelli lemma to the collection of sets
\[
\mathcal B_\ell=\bigcup \left\{\mathcal B(\mathcal K, \delta, T): {T_\ell<T\le T_{\ell+1}}\right\} ,\quad \ell\in \NN,
\]
since $\limsup_{T\rightarrow \infty} \mathcal B(\mathcal K, \delta, T)=\limsup_{\ell\rightarrow \infty} \mathcal B_\ell$.

For each $\ell\in \NN$, one can take a sequence of subsets $\mathcal I_\ell$ of $\mathcal K$ such that
\begin{itemize}
\item $\# \mathcal I_\ell = O_{\mathcal K}\left(\ell^{(n+1)(2n-1)}\right)$;
\item $\mathcal K \subseteq \bigcup_{h\in \mathcal I_\ell} h\mathcal O_\ell$,
where $\mathcal O_\ell=\left\{g\in G_{2n}: \|g\|_\infty, \;\|g^{-1}\|_\infty < 1+\dfrac 1 {\ell} \right\}$
\end{itemize}
(see \cite[Lemma 2.1]{KY2020}).
For any $g\in h\mathcal O_\ell\cap \mathcal K$ and $T_\ell < T \le T_{\ell+1}$, the sets
\[\begin{split}
E^+_{h, \ell}&=\left\{(\vv_1, \ldots, \vv_k)\in (\RR^{2n})^k: \langle \vv_i, \vv_j \rangle \in (a_{ij}, b_{ij}) \right\}\cap \left(hB_{(1+\frac 1 \ell)T_{\ell+1}}(\origin)\right)^k;\\[0.05in]
E^-_{h, \ell}&=\left\{(\vv_1, \ldots, \vv_k)\in (\RR^{2n})^k: \langle \vv_i, \vv_j \rangle \in (a_{ij}, b_{ij}) \right\}\cap \left(hB_{(1-\frac 1 \ell)T_{\ell}}(\origin)\right)^k
\end{split}\]
are a supset and a subset of $E_{g,T}$, respectively. 
It holds that
\[
M_{h,\ell}:=\vol_{2nk}(E^-_{h,\ell})^\delta - \left(\EE(\widetilde{\chi^{}_{E^+_{h,\ell}}}) -\EE(\widetilde{\chi^{}_{E^-_{h,\ell}}})\right)
=O_h\left(\ell^{\delta \alpha (2nk-k(k-1))}\right),
\]
provided that 
\begin{equation}\label{eqn2: main} 
\delta>1 - \frac 1 {\alpha(2nk-k(k-1))}
\end{equation} 
and using \Cref{thm: volume formula} and Proposition~\ref{prop: mean}.
It follows from Lemma~\ref{lem: discrepancy} that 
\[
\limsup_{\ell\rightarrow \infty} \mathcal B_{\ell} \subseteq \limsup_{\ell\rightarrow\infty}
\left(\left\{g\in \mathcal K: D(g\ZZ^{2n}, E_{h, \ell}^-) \ge M_{h, \ell} \right\}\cup\left\{g\in \mathcal K: D(g\ZZ^{2n}, E_{h, \ell}^+) \ge M_{h, \ell} \right\}\right).
\]
Thus, it suffices to show that the sum of measures of sets in (R.H.S) above over $\ell\in \NN$ and $h\in \mathcal I_\ell$ is finite, which directly leads to the theorem by the Borel--Cantelli lemma. More precisely, we want to find conditions for $\delta\in (0,1)$ and $\alpha>1$ under which the sum converges.

Observe that
\[\begin{split}
\mu\left\{g\in \mathcal K: D(g\ZZ^{2n}, E_{h, \ell}^\pm) \ge M_{h, \ell} \right\}
&\le \frac 1 {M_{h,\ell}^2}\int_{G_{2n}/\Gamma_{2n}} D(g\ZZ^{2n}, E_{h,\ell}^\pm)^2 d\mu(g)\\ &\le O_h\left(\ell^{\alpha(2n(2k-1)-2k^2)-2\delta\alpha(2nk-k(k-1))}\right).
\end{split}\]
Hence, there is a constant $C_{\mathcal K}>0$ so that the sum is bounded above by
\[
\sum_{\ell\in \NN} \mu(\mathcal B_\ell) \le C_{\mathcal K} \ell^{\alpha(2n(2k-1)-2k^2)-2\delta\alpha(2nk-k(k-1))+(n+1)(2n-1)},
\]
which converges when
\begin{equation}\label{eqn3: main}
\alpha(2n(2k-1)-2k^2)-2\delta\alpha(2nk-k(k-1))+(n+1)(2n-1)<-1.
\end{equation}

Thus, we need to find $0<\delta<1$ and $\alpha>1$ satisfying both \eqref{eqn2: main} and \eqref{eqn3: main}. Indeed, one can show that $\delta$ which is bounded by
\[
\frac{(n+1)(2n-1)(2nk-k(k-1))+(2n(2k-1)-2k^2)+(2nk-k(k-1))}{(n+1)(2n-1)(2nk-k(k-1))+3(2nk-k(k-1))} < \delta < 1,
\]
and for such a $\delta>0$, $\alpha$ which is bounded by
\[
\max\left\{1,\frac{(n+1)(2n-1)+1}{2\delta(2nk-k(k-1)-(2n(2k-1)-2k^2)}\right\} < \alpha < \frac 1 {(1-\delta)(2nk-k(k-1))},
\]
satisfies both inequalities. For those $\delta\in (0,1)$ and $\alpha>1$, the Borel--Cantelli lemma, together with Corollary~\ref{cor: upper bound} and \Cref{thm: volume formula}, tells us that for almost all $g\in \mathcal K$, it holds that
\[\begin{split}
N_{g,\mathcal I}(T)
&=\EE(\chi^{}_{E_{g,T}}) + O_{g}\left(T^{\delta(2nk-k(k-1))}\right)\\
&=\int_{G_{2n}/\Gamma_{2n}} \widetilde{\chi^{}_{E_{g,T}}} (g'\ZZ^{2n}) d\mu(g')+O_{g}\left(T^{\delta(2nk-k(k-1))}\right)\\
&=\vol_{2nk}(E_{g,T})+O_g\left(T^{2n(k-1)-(k-1)(k-2)}\right)+O_{g}\left(T^{\delta(2nk-k(k-1))}\right)\\
&=c_g\prod_{1\le i<j\le k} (b_{ij}-a_{ij})\cdot T^{2nk-k(k-1)} +O_g\left(T^{2n(k-1)-(k-1)(k-2)}\right)+O_g\left(T^{2nk-k(k-1)-1}\right)\\
&\hspace{2.7in}+O_{g}\left(T^{\delta(2nk-k(k-1))}\right).
\end{split}\]
Since the compact set $\mathcal K$ is arbitrary, the theorem follows.
\end{proof}

\section{Primitive and Congruent Analogs}\label{Sec:Primitive Analog}
\subsection{Proof of Theorem~\ref{thm: quantitative prime}}
Let us consider the primitive version of the rank-$k$ Siegel transform.

\begin{definition}\label{def: primitive Siegel transform}
For a bounded and compactly supported function $F:(\RR^d)^k\rightarrow \RR$, define
\[
\widehat{F}(g\Gamma_{2n})=\widehat{F}(g P(\ZZ^d))=\sum_{\vv_i\in P(\ZZ^d)} F(g\vv_1, \ldots, g\vv_k),
\]
where $P(\ZZ^d)$ is the set of primitive integer vectors in $\RR^d$. 
\end{definition}

Note that the set $P(g\ZZ^d)$ of primitive lattice points in $g\ZZ^d$ is equal to $gP(\ZZ^d)$. Thus, if we take $F$ as the characteristic function of a Borel set in $(\RR^d)^k$, $\widehat{F}(gP(\ZZ^d)$ counts the number of $g\ZZ^d$-primitive lattice points contained in the given Borel set.

In \cite{Hanprim}, the author introduces the incomplete integral formula for $\widehat{F}$.
\begin{theorem}\label{thm: Rogers primitive}
For $d\ge 2$ and $1\le k\le d-1$, it holds that
\[\begin{split}
\int_{\SL_d(\RR)/\SL_d(\ZZ)} \widehat{F}(gP(\ZZ^d))
&=\frac 1 {\zeta(d)^k} \int_{(\RR^d)^k} F(\vv_1, \ldots, \vv_k) d\vv_1 \ldots \vv_k\\
&+\sum_{r=1}^{k-1} \sum_{q\in \NN} \sum_{D\in \widehat{\mathcal D}^k_{r,q}} \widehat{c}_D\int_{(\RR^d)^r} F\left((\vv_1, \ldots, \vv_r) \frac D q \right) d\vv_1 \cdots d\vv_r.
\end{split}\]
Here, $\widehat{\mathcal D}^k_{r,q}$ is the subset of $D\in \mathcal D^k_{r,q}$ such that $$(\ZZ^d)^r \frac D q \cap P(\ZZ^d)^k\neq \emptyset,$$
and in this case, it satisfies that $\widehat{c}_D\le c_D$. In particular, $\widehat c_D\le 1/q^d$.
\end{theorem}

Recall that
\[
E_{g,T}=\left\{(\vv_1, \ldots, \vv_k)\in (\RR^{2n})^k: \langle \vv_i, \vv_j \rangle \in (a_{ij}, b_{ij}) \right\}
\cap (g^{-1}B_T(\origin))^k.
\]
The counting function $\widehat N_{g,\mathcal I}(T)$ in Theorem~\ref{thm: quantitative prime} is given as $\widehat {\chi^{}_{E_{g,T}}}(g\ZZ^{2n})$, where $\widehat{\chi^{}_{E_{g,T}}}$ is the characteristic function of $E_{g,T}$.
The following proposition is deduced from a similar argument for proving Proposition~\ref{prop: mean}, together with Theorem~\ref{thm: Rogers primitive}, especially using the fact that $\widehat c_D \le c_D$ for $D\in \widehat {\mathcal D}_{r,q}^k$.

\begin{proposition}\label{prop: mean primitive} Let $2n\ge (k/2)^2+3$. Let $E_{g,T}$ for $g\in G_{2n}$ be as in Corollary~\ref{cor: upper bound}. It holds that
\[
\int_{G_{2n}/\Gamma_{2n}} \widehat{\chi^{}_{E_{g,T}}} (gP(\ZZ^{2n})) d\mu(g)=\frac 1 {\zeta(2n)^k}\vol_{2nk}(E_{g,T})+O_g\left(T^{2n(k-1)-(k-1)(k-2)}\right).
\]
\end{proposition}

Furthermore, the discrepancy property (Lemma~\ref{lem: discrepancy}) also holds for the rank-$k$ primitive Siegel transform.
\begin{lemma}\label{lem: discrepancy primitive}
Let $F_i:(\RR^{d})^k\rightarrow \RR$, $i=1,2,3$, be bounded and compactly supported functions such that $F_1\le F_2 \le F_3$. 
Denote the discrepancy between the function value $\widetilde{F_i}$ and its average by
\[
D\left(gP(\ZZ^{d}), F_i\right)=\widehat{F_i}(gP(\ZZ^{d})) - \EE(\widehat{F_i}),
\]
where $\EE(\widehat{F_i})$ is the average of $\widehat{F_i}$ over $G_d/\Gamma_d$ with respect to the $G_d$-invariant probability measure $\mu$.
It holds that
\[
D\left(gP(\ZZ^{d}), F_2\right)\le \max\left\{ D\left(gP(\ZZ^{d}), F_1\right), \;D\left(gP(\ZZ^{d}), F_3\right)\right\}.
\]
\end{lemma}

Theorem~\ref{thm: quantitative prime} follows from applying the proof of Theorem~\ref{thm: quantitative} but substituting Theorem~\ref{prop: upper bound} and Corollary~\ref{cor: upper bound} with the following proposition, thus we omit the full proof except below.

\begin{proposition}\label{prop: upper bound primitive}
Assume that $2n\ge k^2+3$.
Let $E\subseteq (\RR^{2n})^k$ be a Jordan-measurable set such that there are the collection $\{E_r\}_{1\le r\le k}$ of Jordan-measurable sets $E_r\subseteq (\RR^{2n})^r$ as in Theorem~\ref{prop: upper bound}.

There is a constant $\widehat C>0$, depending only on $2n$, so that
\[\begin{split}
0\le \int_{G_{2n}/\Gamma_{2n}} \widehat{\chi^{}_E} (gP(\ZZ^{2n}))^2 d\mu(g)
&-\left(\int_{G_{2n}/\Gamma_{2n}} \widehat{\chi^{}_E} (gP(\ZZ^{2n})) d\mu(g)\right)^2\\
&\le \widehat C \max\left\{\vol_{r_1}(E_{r_1})\cdot \vol_{r_2}(E_{r_2}) : 1\le r_1+r_2\le 2k-1\right\}.
\end{split}\]
As a consequence, it follows that
\begin{equation*}
\int_{G_{2n}/\Gamma_{2n}} \widehat{\chi^{}_{E_{h,T}}} (gP(\ZZ^{2n}))^2 d\mu(g)
-\left(\int_{G_{2n}/\Gamma_{2n}} \widehat{\chi^{}_{E_{h,T}}} (gP(\ZZ^{2n})) d\mu(g)\right)^2
= O_g\left(T^{2n(2k-1)-2k^2}\right).
\end{equation*}
\end{proposition}
\begin{proof}
It suffices to show that
\begin{align*}
&\int_{G_{2n}/\Gamma_{2n}} \widehat{\chi^{}_E} (gP(\ZZ^{2n}))^2 d\mu(g)
-\left(\int_{G_{2n}/\Gamma_{2n}} \widehat{\chi^{}_E} (gP(\ZZ^{2n})) d\mu(g)\right)^2\label{eqn2: upper bound}\\
&\hspace{1in}= \sum_{r=1}^{2k-1} \sum_{q\in \NN} \sum_{D\in \widehat{\mathcal D'}^{2k}_{r,q}} 
\widehat c_D\int_{(\RR^{2n})^r} \chi^{}_E\otimes \chi^{}_E\left( (\vv_1, \ldots, \vv_r)\frac D q \right) d\vv_1 \cdots d\vv_r, \nonumber
\end{align*}
where $\widehat{\mathcal D'}^{2k}_{r,q}$ is the set of $r\times 2k$ matrices $D\in \widehat{\mathcal D}^{2k}_{r,q}$ which is not of the following form:
There are $D_1\in \widehat{\mathcal D}^k_{r_1, q_1}$ and $D_2\in\widehat{ \mathcal D}^k_{r_2, q_2}$, where $r_1+r_2=r$ and $\lcm(q_1,q_2)=q$ so that
\[
D=\left(\begin{array}{cc}
\frac q {q_1}D_1 & \\
& \frac q {q_2} D_2\end{array}\right).
\]

To see this, observe that for any $D_1\in \widehat{\mathcal D}^k_{r_1, q_1}$ and $D_2\in\widehat{ \mathcal D}^k_{r_2, q_2}$, the matrices
\[
D_3=\left(\begin{array}{cc}
\frac q {q_1}D_1 & \\
& \frac q {q_2} D_2\end{array}\right)
\quad\text{and}\quad
D_4=\left(\begin{array}{cc}
\frac q {q_2}D_2 & \\
& \frac q {q_1} D_1\end{array}\right)
\]
are in $\widehat{\mathcal D}^{2k}_{r,q}$, where $r=r_1+r_2$ and $q=\lcm(q_1, q_2)$ since if $(\vv_1, \ldots, \vv_k)\in (\ZZ^{2n})^{r_1} \frac {D_1} {q_1} \cap P(\ZZ^{2n})^k$ and $(\vv_{k+1}, \ldots, \vv_{2k})\in (\ZZ^{2n})^{r_2} \frac {D_2} {q_2} \cap P(\ZZ^{2n})^k$,
\[\begin{gathered}
(\vv_1, \ldots, \vv_k, \vv_{k+1}, \ldots, \vv_{2k})\in (\ZZ^{2n})^r \frac {D_3} q \cap P(\ZZ^{2n})^{2k}
\quad\text{and}\\
(\vv_{k+1}, \ldots, \vv_{2k}, \vv_{1}, \ldots, \vv_{k})\in (\ZZ^{2n})^r \frac {D_4} q \cap P(\ZZ^{2n})^{2k}.
\end{gathered}\]
In the spirit of the proof of Rogers' formula based on Riesz representation theorem,  since $D_3$ is a block diagonal matrix, it follows that
\[\begin{split}
\widehat c_{D_3}
&=\lim_{T\rightarrow\infty}\frac
{\#\left\{(\vv_1, \ldots, \vv_{2k})\in (\ZZ^{2n})^r\frac {D_3} q \cap P(\ZZ^{2n})^{2k}\cap B_T(\mathbf 0)^{2k}: (\vv_1, \ldots, \vv_{2k})\;\text{is lin. indep.} \right\}}{\vol_{4nk}(B_T(\mathbf 0))^{2k}}\\
&=\lim_{T\rightarrow\infty}\frac
{\#\left((\ZZ^{2n})^r\frac {D_3} q \cap P(\ZZ^{2n})^{2k}\cap B_T(\mathbf 0)^{2k}\right)}{\vol_{4nk}(B_T(\mathbf 0))^{2k}}\\
&=\lim_{T\rightarrow \infty}\frac {\#\left((\ZZ^{2n})^{r_1}\frac {D_1} {q_1} \cap P(\ZZ^{2n})^{k}\cap B_T(\mathbf 0)^{k}\right)}{\vol_{2nk}(B_T(\mathbf 0))^{k}}\times
\frac{\#\left((\ZZ^{2n})^{r_2}\frac {D_2} {q_2} \cap P(\ZZ^{2n})^{k}\cap B_T(\mathbf 0)^{k}\right)}{\vol_{2nk}(B_T(\mathbf 0))^{k}}\\
&=\widehat c_{D_1} \widehat c_{D_2}
\end{split}\]
and similarly, we obtain that $\widehat c_{D_4}=\widehat c_{D_1}\widehat c_{D_2}$.

The rest of the proof is exactly the same as that of Theorem~\ref{prop: upper bound}: Since $\widehat c_D\le c_D$, one can use the same upper bounds for \textbf{Cases I, II}, and \textbf{III}.
\end{proof}

\subsection{Proof of Theorem~\ref{thm: quantitative congruence}}.
Let $N\in \NN$ and $\vv_0\in \ZZ^{d}$ for which $\gcd(\vv_0, N)=1$.
For a bounded and compactly supported function $F: (\RR^d)^k\rightarrow \RR$, one can define a rank-$k$ Siegel transform associated with the congruence condition $(\vv_0, N)$ as
\[
\mathcal S_{(\vv_0, N)} (F)(g\Gamma_{d}(N))
=\sum_{\vv_i\in (\vv_0+N\ZZ^{d})^k} F(g\vv_1, \ldots, g\vv_k),
\quad \forall g\Gamma_{d}(N)\in G_{d}/\Gamma_{d}(N).
\]
where $\Gamma_{d}(N)$ is the principal congruence subgroup of $\Gamma_{d}=\SL_{d}(\ZZ)$ of level $N$.

\begin{theorem}[{\cite[Theorem 2.13]{AGH2024}}]
Let $d\ge 3$ and $1\le k \le d-1$. Let $N\in \NN$, $\vv_0\in \ZZ^{d}$, and $\mathcal S_{(\vv_0, N)}(F)$ for a bounded and compactly supported function on $(\RR^{d})^k$, as above.

It holds that
\[\begin{split}
    \int_{G_d/\Gamma_d(N)} \mathcal S_{(\vv_0, N)}(F)(g\Gamma_d(N)) d\mu_N
    &= \frac 1 {N^{dk}} \int_{(\RR^d)^k} F(\vv_1, \ldots, \vv_k) d\vv_1 \cdots d\vv_k +\\
    &\sum_{r=1}^{k-1}\sum_{q\in \NN} \sum_{D\in \mathcal C^k_{r,q}} \frac {c_D}{N^{dr}} \int_{(\RR^d)^r}
    F\left((\vv_1, \ldots, \vv_r) \frac {D} {q}\right) d\vv_1 \cdots d\vv_r,
\end{split}\]
where $\mathcal C^k_{r,q}$ is the subset of $\mathcal D^k_{r,q}$ collecting $D$ satisfying that there is $\vv=(v_1, \ldots, v_k)\in \Lambda_D=\RR^r{D}\cap \ZZ^k$ for which
\begin{equation}\label{eqn: congruence admissible}\begin{gathered}
\gcd(v_1, N)=1,
\quad
v_1 \equiv \cdots \equiv v_k \mod N,
\quad\text{and}\\
|v_1|=\min\left(\NN \cap \left\{\vv'\cdot \ve_1: \vv' \in \Lambda_D\right\}\right).
\end{gathered}\end{equation}
Here, $\cdot$ is the usual dot product and the constant $c_D$ for each $D\in \mathcal D^k_{r,q}$ is a constant defined as in \ref{thm: Rogers}.
\end{theorem}

As in the primitive case, let us briefly check congruent analogs of Proposition~\ref{prop: mean}, the discrepancy property (Lemma~\ref{lem: discrepancy}), and Corollary~\ref{cor: upper bound}, and skip the rest of the proof.

Since $\mathcal C^k_{r,q}$ is the subset of $\mathcal D^k_{r,q}$, using the same upper bound in the proof of Proposition~\ref{prop: mean congruence}, we have the similar result that the average of the Siegel transform of $\chi^{}_{E_{g,T}}$ asymptotically converges to the volume of $E_{g,T}$, divided by $N^{2nk}$ as $T$ diverges to infinity. 
 
\begin{proposition}\label{prop: mean congruence} Let $2n\ge (k/2)^2+3$. Let $E_{g,T}$ for $g\in G_{2n}$ be as in Corollary~\ref{cor: upper bound}. It holds that
\[
\int_{G_{2n}/\Gamma_{2n}} \mathcal S_{(\vv_0, N)}({\chi^{}_{E_{g,T}}}) (g\ZZ^{2n}) d\mu(g)=\frac 1 {N^{2nk}}\vol_{2nk}(E_{g,T})+O_g\left(T^{2n(k-1)-(k-1)(k-2)}\right).
\]
\end{proposition}

The discrepancy property in a congruent context also holds as follows.
\begin{lemma}\label{lem: discrepancy congruent}
Let $F_i:(\RR^{d})^k\rightarrow \RR$, $i=1,2,3$, be bounded and compactly supported functions such that $F_1\le F_2 \le F_3$. 
Denote the discrepancy between the function value $\mathcal S_{(\vv_0, N)}({F_i})$ and its average by
\[
D(g\Gamma_d(N), F_i)=D\left(g(\vv_0+N\ZZ^{d}), F_i\right)=\mathcal S_{(\vv_0,N)}({F_i})(g\ZZ^{d}) - \EE(\mathcal S_{(\vv_0,N)}({F_i})),
\]
where $\EE(\mathcal S_{(\vv_0,N)}({F_i}))$ is the average of $\mathcal S_{(\vv_0,N)}({F_i})$ over $G_d/\Gamma_d(N)$ with respect to the $G_d$-invariant probability measure $\mu$.
It holds that
\[
D\left(g(\vv_0+N\ZZ^{d}), F_2\right)\le \max\left\{ D\left(g(\vv_0+N\ZZ^d), F_1\right), \;D\left(g(\vv_0+N\ZZ^{d}), F_3\right)\right\}.
\]
\end{lemma}
 
Finally, let us finish the proof of Theorem~\ref{thm: quantitative congruence} by showing the following proposition.

\begin{proposition}\label{prop: upper bound congruence}
Assume that $2n\ge k^2+3$. Let $N\in \NN$, $\vv_0\in \ZZ^{2n}$ be such that $\gcd(\vv_0, N)=1$.
Let $E\subseteq (\RR^{2n})^k$ be a Jordan-measurable set such that there are the collection $\{E_r\}_{1\le r\le k}$ of Jordan-measurable sets $E_r\subseteq (\RR^{2n})^r$ as in Theorem~\ref{prop: upper bound}.

There is a constant $C_{cong}>0$, depending only on $2n$, so that
\[\begin{split}
0&\le \int_{G_{2n}/\Gamma_{2n}} \left(\mathcal S_{(\vv_0, N)}({\chi^{}_E}) (g\ZZ^{2n})\right)^2 d\mu(g)
-\left(\int_{G_{2n}/\Gamma_{2n}} \mathcal S_{(\vv_0, N)}({\chi^{}_E}) (g\ZZ^{2n}) d\mu(g)\right)^2\\
&\le C_{cong} \max\left\{\vol_{r_1}(E_{r_1})\cdot \vol_{r_2}(E_{r_2}) : 1\le r_1+r_2\le 2k-1\right\}.
\end{split}\]
As a consequence, it follows that
\begin{equation*}\begin{split}
&\int_{G_{2n}/\Gamma_{2n}} \left(\mathcal S_{(\vv_0,N)}(\chi^{}_{E_{h,T}}) (g\ZZ^{2n})\right)^2 d\mu(g)
-\left(\int_{G_{2n}/\Gamma_{2n}} \mathcal S_{(\vv_0,N)}(\chi^{}_{E_{h,T}}) (g\ZZ^{2n}) d\mu(g)\right)^2\\
&\hspace{0.2in}= O_g\left(T^{2n(2k-1)-2k^2}\right).
\end{split}\end{equation*}
\end{proposition}
\begin{proof}
As in the proof of Proposition~\ref{prop: upper bound primitive}, it suffices to show that all integral terms in
\[
\left(\int_{G_{2n}/\Gamma_{2n}} \mathcal S_{(\vv_0, N)}({\chi^{}_E}) (g\ZZ^{2n}) d\mu(g)\right)^2
\]
disappear: For $D_1\in \mathcal C^k_{r_1,q_1}$ and $D_2\in \mathcal C^k_{r_2,q_2}$ (note that this has a full generality if we consider $\mathcal C^k_{k,1}=\{\Id_k\}$), we need to show that
\[
D_3=\left(\begin{array}{cc}
\frac q {q_1}D_1 & \\
& \frac q {q_2}D_2\end{array}\right)
\quad\text{and}\quad
D_4=\left(\begin{array}{cc}
\frac q {q_2}D_2 & \\
& \frac q {q_1}D_1\end{array}\right)\in \mathcal C^{2k}_{r,q},
\] 
where $r=r_1+r_2$ and $q=\lcm(q_1,q_2)$.
Let $\vv_1=(v_1, \ldots, v_k)\in \Lambda_{D_1}$ be a vector satisfying conditions in \eqref{eqn: congruence admissible}. Since $\Lambda_{D_3}=\Lambda_{D_1}\oplus \Lambda_{D_2}$, it is sufficient to find a vector $\vv_2$ in $\Lambda_{D_2}$ for which each component of $\vv_2$ is congruent to $v_1$ modulo $N$ for $1\le i\le k$. 
Let $\vw=(w_1, \ldots, w_k)\in \Lambda_{D_2}$ be a vector satisfying conditions in \eqref{eqn: congruence admissible}. Since $\gcd(w_1, N)=1$ as well as $\gcd(v_1, N)=1$, there is $m\in \ZZ$ such that $w_1m=v_1$ modulo $N$. Then $\vv_2=c\vw\in \Lambda_{D_2}$ is the vector that we want. Since $\vv_3=(\vv_1, \vv_2)\in \Lambda_{D_3}$ satisfies the conditions in \eqref{eqn: congruence admissible}, it follows that $D_3\in \mathcal C^{2k}_{r,q}$, and similarly $D_4\in \mathcal C^{2k}_{r,q}$.
We already confirmed that $c_{D_3}=c_{D_4}=c_{D_1}c_{D_2}$ in the proof of Proposition~\ref{prop: upper bound}, and hence we have
\[
\frac {c_{D_3}}{N^{rk}}=\frac {c_{D_4}}{N^{rk}}=\frac {c_{D_1}}{N^{r_1k}}\frac {c_{D_2}}{N^{r_2k}}.
\]
\end{proof}

\section*{Appendix}
In the Appendix, we recover the proof of the well-known fact that $\Sp(2n,\RR)$ is a maximal connected subgroup of $SL_{2n}(\RR)$, by proving that $\Lie{sp}(2n,\RR)$ is a maximal subalgebra of $\Lie{sl}_{2n}(\RR)$. 

Let us denote an element of $\Lie{sl}_{2n}(\RR)$ a block matrix
\[
\left(\begin{array}{cc}
A & B \\
C & D\end{array}\right),
\]
where $A,\;B,\;C$ and $D\in \Mat_{n}(\RR)$ with $\tr A+\tr D=0$. The Lie subalgebra $\Lie{sp}(2n,\RR)$ is defined as
\[
\Lie{sp}(2n,\RR)=\left\{\left(\begin{array}{cc}
A & B \\
C & D\end{array}\right)\in \Lie{sl}_{2n}(\RR): \begin{array}{c}
D=-{\tp{A}}\;\text{and}\\
B={\tp{B}},\; C={\tp{C}}
\end{array}\right\}.
\]

Recall the Cartan decomposition of $\Lie{sp}(2n,\RR)$ with respect to the involution $X\mapsto -{\tp{X}}$, which is given as $\Lie{sp}(2n,\RR)=\Lie{k}\oplus \Lie{p}$, where 
\[\begin{split}
\Lie{k}&=\left\{\left(\begin{array}{cc}
U & V \\
-V & U \end{array}\right): \begin{array}{c}
{\tp{U}}=-U\\
{\tp{V}}=V\end{array}\right\}\simeq \Lie{u}(n),\\
\Lie{p}&=\left\{\left(\begin{array}{cc}
U & V \\
V & -U\end{array}\right):\begin{array}{c}
{\tp{U}}=U\\
{\tp{V}}=V\end{array}\right\},
\end{split}\]
and the maximal torus in this decomposition is 
\[
\Lie{a}=\left\{\diag(a_1, \ldots, a_n, -a_1, \ldots, -a_n): a_1, \ldots, a_n\in \RR\right\}.
\]

Define $f_i\in \Lie{a}^*$ for $ i=1, \ldots, n$ as
\[
f_i\left(\diag(a_1, \ldots, a_n, -a_1, \ldots, -a_n)\right)=a_i.
\]

\begin{proposition*}
The restricted roots for $(\Lie{sp}(2n,\RR), \Lie{a})$ are $f_i-f_j$ for $1\le i\neq j \le n$ and $f_i+f_j$, $-f_i-f_j$ for $1\le i\le j \le n$. Their root spaces are as follows.
\[\begin{gathered}
\Lie{sp}(2n,\RR)_{f_i-f_j}=\RR.\left(\begin{array}{c|c} 
E_{ij} & \\
\hline
& -E_{ji}\end{array}\right),
\quad
\Lie{sp}(2n,\RR)_{f_i+f_j}=\RR.\left(\begin{array}{c|c}
\hspace{0.15in}& E_{ij}+E_{ji}\\
\hline
 & \end{array}\right),
 \;\text{and}\\
\Lie{sp}(2n,\RR)_{-f_i-f_j}=\RR.\left(\begin{array}{c|c}
 & \\
 \hline
 E_{ij}+E_{ji} & \hspace{0.15in} \end{array}\right).
\end{gathered}\]
Here, $\{E_{ij}\}_{1\le i,j\le n}$ is the standard basis for $\Mat_n(\RR)$. 
\end{proposition*}

Since $\Lie{sp}(2n,\RR)$ is a semisimple Lie algebra, any of its Lie algebra representations has a weight decomposition. We are particularly interested in the weight decomposition for $\Lie{sl}_{2n}(\RR)$ with respect to the adjoint representation of $\Lie{sp}(2n,\RR)$. 

Let us denote by the complement of $\Lie{sp}(2n,\RR)$ in $\Lie{sl}_{2n}(\RR)$ as
\[
W=\left\{\left(\begin{array}{cc}
A & B \\
C & D\end{array}\right)\in \Lie{sl}_{2n}(\RR): \begin{array}{c}
D={\tp{A}}\;\text{and}\\
B=-{\tp{B}},\; C=-{\tp{C}}\end{array}\right\}.
\]
Note that $W$ is an $\Lie{sp}(2n,\RR)$-invariant subspace, but not a Lie subalgebra. Together with the proposition above, it suffices to examine the weight decomposition for $W$.

\begin{proposition*}
The weights for the adjoint representation $W$ of $\Lie{sp}(2n,\RR)$ are 
\[
f_i-f_j\;\text{for}\; 1\le i\neq j \le n,
\quad\text{and}\quad 
f_i+f_j, \;-f_i-f_j\;\text{for}\;1\le i < j \le n.
\]
Their weight spaces are described below:
\[\begin{gathered}
\Lie{m}=W_0=\bigoplus_{i=1}^{n-1} \RR.\left(\begin{array}{c|c}
E_{ii}-E_{i+1,i+1} & \\
\hline
& E_{ii}-E_{i+1,i+1}\end{array}\right);\\
W_{f_i-f_j}=\RR.\left(\begin{array}{c|c}
E_{ij} & \\
\hline
& E_{ji}\end{array}\right),
\quad
W_{f_i+f_j}=\RR.\left(\begin{array}{c|c}
\hspace{0.15in} & E_{ij}-E_{ji}\\
\hline
&
\end{array}\right), \;\text{and}\\
W_{-f_i-f_j}=\RR.\left(\begin{array}{c|c}
& \\
\hline
E_{ij}-E_{ji} & \hspace{0.15in}
\end{array}\right).
\end{gathered}\]
\end{proposition*}

\begin{theorem*}
The subspace $W$ is an $\Lie{sp}(2n,\RR)$-invariant irreducible subspace. As a consequence, there is no proper subalgebra in $\Lie{sl}_{2n}(\RR)$ which properly contains $\Lie{sp}(2n,\RR)$. 
\end{theorem*}
\begin{proof}
It suffices to show that one can transfer weight vectors of one weight of $W$ to weight vectors of another weight in a finite number of steps by the adjoint action of $\Lie{sp}(2n,\RR)$. This can be easily checked from the following list of facts.
\begin{itemize}
\item Between weight vectors of weights of the form $f_i-f_j$ ($1\le i\neq j\le n$),
\[\begin{gathered}
\left[\left(\begin{array}{cc}
E_{j\ell}&\\
& -E_{\ell j}
\end{array}\right),\left(\begin{array}{cc}
E_{ij}&\\
&E_{ji}
\end{array}\right)\right]=-\left(\begin{array}{cc}
E_{i\ell}&\\
&E_{\ell i}
\end{array}\right)\quad\text{and}\\
\left[\left(\begin{array}{cc}
E_{\ell i}&\\
&-E_{i\ell}
\end{array}\right),\left(\begin{array}{cc}
E_{ij}&\\
&E_{ji}
\end{array}\right)\right]=\left(\begin{array}{cc}
E_{\ell j}&\\
&E_{j \ell}
\end{array}\right)\quad\text{for}\;\ell\neq i,\;j.
\end{gathered}\]
\item Between weight vectors of weights $f_i-f_j$ ($1\le i\neq j\le n$) and vectors in $W_0$,
\[\begin{gathered}
\left[\left(\begin{array}{cc}
E_{ji}&\\
&-E_{ij}
\end{array}\right),\left(\begin{array}{cc}
E_{ij}&\\
&E_{ji}
\end{array}\right)\right]=-\left(\begin{array}{cc}
E_{ii}-E_{jj}&\\
&E_{ii}-E_{jj}
\end{array}\right)\quad\text{and}\\
\left[\left(\begin{array}{cc}
E_{ij}&\\
&-E_{ji}
\end{array}\right),\left(\begin{array}{cc}
E_{ii}-E_{jj}&\\
&E_{ii}-E_{jj}
\end{array}\right)\right]=-2\left(\begin{array}{cc}
E_{ij}&\\
&E_{ji}
\end{array}\right).
\end{gathered}\]
\item Between weight vectors of weights of the form $f_i+f_j$ ($1\le i<j\le n$),
\[\begin{gathered}
\left[\left(\begin{array}{cc}
E_{\ell j}&\\
&-E_{j\ell}
\end{array}\right),\left(\begin{array}{cc}
& E_{ij}-E_{ji}\\
O &
\end{array}\right)\right]=\left(\begin{array}{cc}
&E_{i\ell}-E_{\ell i}\\
O &
\end{array}\right)\quad\text{and}\\
\left[\left(\begin{array}{cc}
E_{\ell i}&\\
&-E_{i \ell}
\end{array}\right),\left(\begin{array}{cc}
&E_{ij}-E_{ji}\\
O &
\end{array}\right)\right]=\left(\begin{array}{cc}
&E_{\ell j}-E_{j\ell}\\
O &
\end{array}\right)\quad\text{for}\;\ell\neq i,\;j.
\end{gathered}\]
\item Between weight vectors of weights of the form $-f_i-f_j$ ($1\le i<j\le n$),
\[\begin{gathered}
\left[\left(\begin{array}{cc}
E_{j\ell}&\\
&-E_{\ell j}
\end{array}\right),\left(\begin{array}{cc}
&O\\
E_{ij}=E_{ji}&
\end{array}\right)\right]=-\left(\begin{array}{cc}
& O\\
E_{i\ell}-E_{\ell i}&
\end{array}\right)\quad\text{and}\\
\left[\left(\begin{array}{cc}
E_{i\ell}&\\
&-E_{\ell i}
\end{array}\right),\left(\begin{array}{cc}
& O\\
E_{ij}-E_{ji}&
\end{array}\right)\right]=\left(\begin{array}{cc}
& O\\
E_{j\ell}-E_{\ell j}&
\end{array}\right)\quad\text{for}\;\ell\neq i,\;j.
\end{gathered}\]
\item Between weight vectors of weights $f_i-f_j$ and $f_i+f_j$ ($i\neq j$),
\[\begin{gathered}
\left[\left(\begin{array}{cc}
& E_{jj}\\
O &
\end{array}\right),\left(\begin{array}{cc}
E_{ij}&\\
&E_{ji}
\end{array}\right)\right]=-\left(\begin{array}{cc}
&E_{ij}-E_{ji}\\
O &
\end{array}\right)\quad\text{and}\\
\left[\left(\begin{array}{cc}
&O \\
E_{jj}&
\end{array}\right),\left(\begin{array}{cc}
&E_{ij}-E_{ji}\\
O & 
\end{array}\right)\right]=-\left(\begin{array}{cc}
E_{ij}&\\
&E_{ji}
\end{array}\right).
\end{gathered}\]
\item Between weight vectors of weights $f_i-f_j$ and $-f_i-f_j$ ($i\neq j$),
\[\begin{gathered}
\left[\left(\begin{array}{cc}
& O\\
E_{ii}&
\end{array}\right),\left(\begin{array}{cc}
E_{ij}&\\
&E_{ji}
\end{array}\right)\right]=\left(\begin{array}{cc}
& O\\
E_{ij}-E_{ji}&
\end{array}\right)\quad\text{and}\\
\left[\left(\begin{array}{cc}
&E_{ii}\\
O &
\end{array}\right),\left(\begin{array}{cc}
& O\\
E_{ij}-E_{ji}&
\end{array}\right)\right]=\left(\begin{array}{cc}
E_{ij}&\\
&E_{ji}
\end{array}\right).
\end{gathered}\]
\end{itemize}

\vspace{0.15in}
Now, let $\Lie{g}\le \Lie{sl}_{2n}(\RR)$ be a subalgebra containing $\Lie{sp}(2n,\RR)$. Then $\Lie{g}$ is invariant under the adjoint action of $\Lie{sp}(2n,\RR)$, so that $\Lie{g}$ can be denoted by the direct sum of two $\Lie{sp}(2n,\RR)$-invariant subspaces $\Lie{sp}(2n,\RR)\oplus W'$, where $W'\subseteq W$. 
Since $W$ is $\Lie{sp}(2n,\RR)$-invariant irreducible, if $W'\neq 0$, then $W'=W$. Hence either $\Lie{g}=\Lie{sp}(2n,\RR)$ or $\Lie{sl}_{2n}(\RR)$.
\end{proof}


\end{document}